\documentclass[12pt]{amsart}
\usepackage[margin = 1.0in]{geometry}

\usepackage{tikz}
\usepackage{tikz-cd}
\usepackage{url, hyperref}
\usetikzlibrary{shapes.arrows}
\usetikzlibrary{decorations.pathreplacing}
\usepackage[all]{xy}

\usepackage{amsmath}
\usepackage{amssymb}
\usepackage{enumitem}
\usepackage{graphicx}
\usepackage{mathdots}
\usepackage{color}
\usepackage{diagbox}
\usepackage{array, makecell}
\usepackage{rotating}
\usepackage{amsmath}
\usepackage{tikz-cd}

\usepackage{amsthm}
\newtheorem{definition}{Definition}[subsection]
\newtheorem{theorem}[definition]{Theorem}
\newtheorem{example}[definition]{Example}
\newtheorem{proposition}[definition]{Proposition}
\newtheorem{corollary}[definition]{Corollary}
\newtheorem{lemma}[definition]{Lemma}
\newtheorem{remark}[definition]{Remark}

\DeclareMathOperator{\Adm}{Adm}
\DeclareMathOperator{\inbox}{box}

\DeclareMathOperator{\Fl}{Fl}
\DeclareMathOperator{\Gr}{Gr}
\DeclareMathOperator{\LG}{LG}
\DeclareMathOperator{\NE}{NE}
\DeclareMathOperator{\pr}{pr}
\DeclareMathOperator{\OG}{OG}
\DeclareMathOperator{\SO}{SO}
\DeclareMathOperator{\ssp}{sp}
\DeclareMathOperator{\SW}{SW}
\DeclareMathOperator{\rk}{rk}
\DeclareMathOperator{\Spec}{Spec}
\DeclareMathOperator{\mAdm}{mAdm}
\DeclareMathOperator{\SC}{sc}

\newcommand{\bA}{\mathbb{A}}
\newcommand{\bC}{\mathbb{C}}
\newcommand{\cC}{\mathcal{C}}
\newcommand{\cD}{\mathcal{D}}
\newcommand{\cF}{\mathcal{F}}
\newcommand{\cM}{\mathcal{M}}
\newcommand{\bP}{\mathbb{P}}
\newcommand{\cO}{\mathcal{O}}
\newcommand{\cP}{\mathcal{P}}
\newcommand{\bR}{\mathbb{R}}
\newcommand{\cT}{\mathcal{T}}
\newcommand{\cX}{\mathcal{X}}
\newcommand{\bZ}{\mathbb{Z}}
\newcommand{\cZ}{\mathcal{Z}}

\newcommand{\wtn}{\widetilde{[n]}}
\newcommand{\barone}{\overline{1}}
\newcommand{\barj}{\overline{j}}
\newcommand{\bark}{\overline{k}}
\newcommand{\barn}{\overline{n}}
\newcommand{\barq}{\overline{q}}
\newcommand{\vectt}{\overrightarrow{t}}

\title{Delta-matroids and toric degenerations in $\OG(n,2n+1)$}

\author{Chen Chen}
\address{University of Illinois Urbana-Champaign, Department of Mathematics, 273 Altgeld Hall, 1409 W. Green Street
\hfill \newline\texttt{}
 \indent Urbana, IL 61801} \email{{\tt Chenc14@illinois.edu}}

 \author{Carl Lian}
\address{Washington University in St. Louis, Department of Mathematics, 1 Brookings Drive
\hfill \newline\texttt{}
 \indent  St. Louis, MO 63130} \email{{\tt clian@wustl.edu}}

 \date{\today}

\usepackage{graphicx}

\begin{document}

\begin{abstract}
    We construct an explicit, embedded degeneration of the general torus orbit closure in the maximal orthogonal Grassmannian $\OG(n,2n+1)$ into a union of Richardson varieties. In particular, we deduce a formula for the cohomology class of the torus orbit closure, as a sum of products of Schubert classes. The moment map images of the degenerate pieces are the base polytopes of their underlying delta-matroids, and give a polyhedral decomposition of the unit hypercube, which had previously been studied by Chen-Sanchez-Veliz-Ying.
\end{abstract}

\maketitle

\setcounter{tocdepth}{1}

\tableofcontents

\section{Introduction}

Let $w_1,\ldots,w_n\in\bC^r$ be a collection of vectors. The associated \emph{matroid} $M$ is the data of the $r$-element subsets of $S\subset [n]$ for which $\{w_q:q\in S\}$ is a basis for $\bC^r$. Many interesting combinatorial invariants of the matroid may be encoded geometrically in the associated torus orbit closure $Z\subset \Gr(r,n)$ \cite{speyer}, defined as follows. Let $A$ be the $r\times n$ matrix whose columns are given by the $w_j$. Let $\Lambda\in\Gr(r,n)$ be the row-span of $A$, assumed to be of dimension $r$. The $n$-dimensional torus $T=(\bC^{*})^n$ acts on matrices $A$ by scaling columns. Then, $Z=\overline{T\cdot \Lambda}$ is the closure of the orbit of $\Lambda$ under the induced $T$-action on $\Gr(r,n)$.

It is therefore of combinatorial and geometric interest to determine the cohomology classes (more generally, K-theory classes) of torus orbit closures $Z\subset\Gr(r,n)$. For the torus orbit closures of a \emph{general} linear subspace $\Lambda$ (associated to the \emph{uniform} matroid $M$), Berget-Fink proved via equivariant localization \cite{bf} that
\begin{equation}\label{eq:bf}
    [Z]=\sum_{\mu\subset(n-r-1)^{r-1}}\sigma_{\mu}\sigma_{\overline{\mu}}\in H^{2(n-r-1)(r-1)}(\Gr(r,n)).
\end{equation}
Here, $\overline{\mu}$ denotes the complement of the partition $\mu$ in a rectangle with $n-r-1$ columns and $r-1$ rows.

The formula \eqref{eq:bf} also follows from an earlier result of Anderson-Tymoczko \cite{at} of the same shape for general torus orbits (more generally, Hessenberg varieties) in the full flag variety $\Fl(n)$, obtained using a degeneracy locus formula of Fulton \cite{fulton}. Klyachko \cite{klyachko} had previously found a different formula for the class $[Z]$, expressed additively in terms of the Schubert basis. A comparison of these formulas on $\Gr(r,n)$ via Coskun's geometric Littlewood Richardson rule \cite{coskun_lr} is given in \cite{lian_lr}. A combinatorial interpretation for Klyachko's coefficients on $\Fl(n)$ has recently been obtained by Nadeau-Tewari \cite{nt}.

The form of \eqref{eq:bf} suggests that one may be able to obtain a different proof by degeneration. Namely, the formula would follow from the existence of an embedded degeneration of $Z$ into a union of generically transverse intersections of two Schubert varieties (Richardson varieties) of classes $\sigma_\mu,\sigma_{\overline{\mu}}$. Indeed, such a degeneration was constructed explicitly in \cite{lian_hhmp} in the full flag variety setting, also re-proving the formula of Anderson-Tymoczko. See also \cite{lian_pr} for connections between these degenerations and curve-counting questions in enumerative geometry. Via the toric moment map, the combinatorial shadow of this degeneration is Harada-Horiguchi-Masuda-Park's polyhedral decomposition of the permutohedron \cite{hhmp}.

The purpose of this paper is to extend this story to type B. The type B analog of a matroid is a \emph{delta-matroid}. See \cite{efls,els,larson} for recent work on the combinatorics and geometry of delta-matroids. Let $\bC^{2n+1}$ be an odd-dimensional complex vector space, equipped with a non-degenerate symmetric quadratic form, and let $\OG(n,2n+1)$ be the orthogonal Grassmannian of (maximal) isotropic subspaces of dimension $n$. Let $\Lambda\in \OG(n,2n+1)$ be a point, given by the row span of a $n\times (2n+1)$-matrix $A_\Lambda$. Then, the (realizable) delta-matroid associated to a point $\Lambda$ is given by the data of the admissible (Definition \ref{def:admissible}) $n$-element subsets of independent column vectors of $A_\Lambda$.

The vector space $\bC^{2n+1}$, and thus the orthogonal Grassmannian $\OG(n,2n+1)$, admit a natural action of the $n$-dimensional torus $T$, and it is again of interest to determine the class of a general torus orbit $[Z]$ in $H^{*}(\OG(n,2n+1))$.\footnote{In \cite{els}, a $K$-theory class $y(\cD)$ on $\OG(n,2n+1)$ is associated to a delta-matroid $\cD$. If $\cD=\cD(\Lambda)$ is the realizable delta-matroid associated to $\Lambda\in\OG(n,2n+1)$, then $y(\cD)$ does \emph{not} always agree with the class of the structure sheaf of the associated torus-orbit closure $[\cO_{\overline{T\cdot\cD(\Lambda)}}]$, see \cite[Proposition 2.3]{els}, in contrast to the setting of ordinary matroids \cite{speyer}. However, these classes do agree when $\Lambda$ is general, which is the main case of interest in this paper.} We prove:

\begin{theorem}\label{thm:class_formula}
    Let $\Lambda\in \OG(n,2n+1)$ be a general point, and let $Z_\Lambda=\overline{T\cdot \Lambda}$ be the closure of the orbit of $T\cong(\bC^{*})^n$. Then,
    \begin{equation*}
        [Z_\Lambda]=\sum_{I\subset[n-1]}\sigma_I\sigma_{I^c}\in H^{n(n-1)}(\OG(n,2n+1)).
    \end{equation*}
\end{theorem}
Here, $I^c$ denotes the complement of $I\subset[n-1]$. See \S\ref{sec:OG} for notation for Schubert classes and Schubert varieties. While other proofs of Theorem \ref{thm:class_formula} along the lines of the earlier work in type A are presumably possible, our approach is by degeneration. We prove:

\begin{theorem}\label{thm:degen}
There exists an explicit, $T$-invariant degeneration of $Z_\Lambda$ into a union of Richardson varieties
\begin{equation*}
    \Sigma_{I,I^c}:=\Sigma^F_I\cap \Sigma^{F'}_{I^c},
\end{equation*}
as subschemes of $\OG(n,2n+1)$, each appearing with multiplicity 1.
\end{theorem}
In particular, Theorem \ref{thm:class_formula} also holds $T$-equivariantly. Our degeneration is constructed in \S\ref{sec:degen}. It is, in some sense, as simple as possible: it proceeds iteratively by sending entries of a matrix whose row span is $\Lambda$ to 0, one at a time. The intermediate steps go through torus orbits of subspaces lying in Richardson varieties associated to \emph{allowed pairs} (Definition \ref{def:allowed_pairs}). While the approach is similar to that of \cite{lian_hhmp}, in type B, one needs more care to ensure that the degeneration of matrix entries is compatible with the underlying quadratic form. In \S\ref{sec:charts}, we identify coordinates for the intermediate Richardson varieties that appear, which allow us to carry out the degeneration in explicit terms while controlling the interaction with the quadratic form.

While the explicit equations allow us to identify limit components of our degeneration, the inexistence of further components is less apparent. This is established via the moment map: we show in \S\ref{sec:moment_polytopes} that the moment polytopes (base polytopes of the associated delta-matroids) of the Richardson varieties $\Sigma_{I,I^c}$ cover that of the general torus orbit. Here, the picture is again as simple as possible: the base polytope of the uniform delta-matroid is the unit hypercube $[0,1]^n$, and the pieces are cut out by hyperplanes. The polyhedral decomposition we obtain had previously been studied by Chen-Sanchez-Veliz-Ying \cite[Proposition 6.1]{csvy} in the context of ``lattice path delta matroids.''

That the limit components each appear with multiplicity 1 is also deduced on the level of moment polytopes. Unlike in the case of ordinary Grassmannians, this is not immediate, see Remark \ref{rem:not_always_reduced}.

In the appendix, we record some well-known foundational results on $\OG(n,2n+1)$ and delta-matroids. In particular, we provide a proof of the equivalence of two definitions of the base polytope of the matroid associated to $\Lambda\in\OG(n,2n+1)$ in Theorem \ref{thm:polytopes_same}. 

In February 2025, we learned, from a talk of Melissa Sherman-Bennett at the IAS workshop ``Combinatorics of Enumerative Geometry,'' of the forthcoming work of Knutson--Sanchez--Sherman-Bennett \cite{kssb} on degenerations of torus orbits in flag varieties $G/B$ of arbitrary Lie type. They construct, for every Coxeter element $w\in W$ of the associated Weyl group, a degeneration of the general torus orbit closure into a union of Richardson varieties, and study the corresponding moment map images. In particular, they obtain families of different formulas for the torus orbit closures, which push forward to formulas on the partial flag varieties $G/P$.

In type B, it is highly plausible, but as far as we are aware, not obvious, that (at least) one of their degenerations on the full type B flag variety $\SO(2n+1)/B$ parametrizing maximal flags of isotropic subspaces pushes forward to the one that we have constructed on $\OG(n,2n+1)$. Their degenerations are valid in much greater generality than ours, but as we understand, are constructed more abstractly. We have tried to emphasize in this work the explicit nature of our degeneration; it would be interesting to match it to one of theirs. We thank the authors of \cite{kssb} for communicating their results and encouraging us to complete this paper.

\subsection{Conventions}\label{sec:conventions}

\begin{itemize}

\item If $S$ is a subset, then its cardinality is denoted $|S|$.
\item $[n]$ denotes the set $\{1,\ldots,n\}$, and $\overline{[n]}$ denotes the set $\{\bar1,\ldots,\barn\}$. We also write $\wtn:=[n]\cup\overline{[n]}$ and $\wtn^0=\wtn\cup\{0\}$.
\item If $a,b$ are positive integers, then we often write $[a,b]$ to mean the set of \emph{integers} in the interval $[a,b]$, and $\overline{[a,b]}$ to denote the set $\{\overline{a},\overline{a+1},\ldots,\overline{b}\}$.
\item The symbol $\overline{\phantom{x}}$ denotes the natural involution on the set $\wtn^0$, where $\overline{0}=0$, as well as the natural involution on the set of \emph{subsets} $S\subset \wtn^0$. For example, if $S=\{0,1,\overline{2}\}$, then $\overline{S}=\{0,\bar1,2\}$.
\item We use angle brackets $\langle\cdot\rangle$ to denote linear span of vectors.
\item The entry of a matrix $A$ in row $p$ and and column $q$ is denoted $a^p_q$. The column set of $A$ will typically be identified with $\wtn^0$, in which case $q$ is an element of $\wtn^0$.
\item If $A$ is a matrix with column set $\wtn^0$ and $S\subset\wtn^0$ is a subset, then $A_S$ denotes the sub-matrix of $A$ consisting of columns indexed by $S$.
\item If $X$ is an irreducible quasi-projective scheme over $\bC$, then a ``general point'' is a closed point lying in a (possibly unspecified) non-empty Zariski open subset. The term ``generic point'' is reserved for the unique point (not closed unless $X$ is itself a point) which lies in every non-empty Zariski open subset of $X$. The same goes for the general and generic fibers of a morphism $f:Y\to X$.
\end{itemize}

\subsection{Acknowledgments}

G.C. was supported by the Masters' Enrollment Reinvestment Program at Tufts University. C.L. has been supported by NSF Postdoctoral Fellowship DMS-2001976, an AMS-Simons travel grant, and the Summer Scholars program at Tufts University. We thank Matt Larson for comments on a draft of this work.

\section{Preliminaries}

\subsection{Orthogonal Grassmannians}\label{sec:OG}

Let $n\ge 2$ be an integer. Let $\bC^{2n+1}$ be a vector space, for which we fix a basis
\begin{equation*}
    e_1,\ldots,e_n,e_0,e_{\barn},\ldots,e_{\barone}
\end{equation*}
Throughout, we refer to this basis as the \emph{standard basis} of $\bC^{2n+1}$. The \emph{coordinates} of a vector $x\in \bC^{2n+1}$ always refer to the coordinates in the standard basis, and matrices whose row vectors are elements of $\bC^{2n+1}$ expressed in standard coordinates are understood to have columns indexed by the set $\wtn^0$.

For any subset $S\subset\wtn^0$, let $H_S\subset \bC^{2n+1}$ denote the span of the basis vectors $e_j$ for $j$ lying in the \emph{complement} of $S$. That is, $H_S$ is the subspace of vectors which, when expressed in the standard basis, are zero in the coordinates indexed by $S$. In particular, the codimension of $H_S$ in $\bC^{2n+1}$ is $|S|$. We also write $H_j=H_{\{j\}}$.

We equip $\bC^{2n+1}$ with the non-degenerate symmetric form defined by
\begin{equation*}
    y\cdot z=2y_0z_0+(y_1z_{\bar1}+y_{\bar1}z_1)+\cdots+(y_nz_{\barn}+y_{\barn}z_n),
\end{equation*}
where the $y_q,z_q$, $q\in\wtn^0$ are the coordinates of $y,z$, respectively. Let $\OG(n,2n+1)$ be the maximal orthogonal Grassmannian of $n$-dimensional isotropic subspaces of $\bC^{2n+1}$. The $\bC$-scheme $\OG(n,2n+1)$ is smooth, projective, and irreducible of dimension $\frac{n(n+1)}{2}$. The orthogonal Grassmannian is naturally included in the usual Grassmannian $\Gr(n,2n+1)$ of $n$-planes in $\bC^{2n+1}$, which in turn is embedded in the projective space $\bP^{\binom{2n+1}{n}-1}$ via the Pl\"{u}cker embedding:
\begin{equation*}
    \OG(n,2n+1)\hookrightarrow \Gr(n,2n+1)\hookrightarrow \bP(\wedge^n(\bC^{2n+1}))\cong \bP^{\binom{2n+1}{n}-1}.
\end{equation*}

The $n$-dimensional torus $T=(\bC^{*})^n$ acts linearly on $\bC^{2n+1}$ by:
\begin{align*}
    (t_1,\ldots,t_n)\cdot e_q&=t_qe_q,\\
    (t_1,\ldots,t_n)\cdot e_{\barq}&=t_q^{-1}e_{\barq}
\end{align*}
for $j=1,\ldots,n$, and $(t_1,\ldots,t_n)\cdot e_0=e_0$. The action of $T$ on $\bC^{2n+1}$ induces a generically free action of $T$ on $\OG(n,2n+1)$.

Let $\Lambda\in\OG(n,2n+1)$ be a point. In this paper, we will be primarily be interested in the torus orbit closures
\begin{equation*}
    Z_\Lambda:=\overline {T\cdot\Lambda}\subset\OG(n,2n+1).
\end{equation*}
and their classes $[Z_\Lambda]\in H^{*}(\OG(n,2n+1))$. If $\Lambda$ is a general point, then $Z_\Lambda$ is irreducible of dimension $n$.

\subsection{Schubert and Richardson varieties}\label{sec:schubert}

We next recall the definitions of Schubert varieties in $\OG(n,2n+1)$. We essentially follow the indexing of \cite{coskun_rigid}. For $q=0,\ldots,n$, define the isotropic subspaces
\begin{align*}
    F_q&=\langle e_1,\ldots,e_q\rangle,\\
    F'_q&=\langle e_{\barone},\ldots,e_{\barq}\rangle.
\end{align*}
of $\bC^{2n+1}$. When $q=0$, the subspaces $F_q,F'_q$ are both zero. Define the complete isotropic flags
\begin{align*}
    F&=(0=F_0\subset F_1\subset\cdots\subset F_n\subset F_n^\perp\subset\cdots\subset F_1^\perp\subset F_0^\perp=\bC^{2n+1}),\\
    F'&=(0=F'_0\subset F'_1\subset\cdots\subset F'_n\subset (F'_n)^\perp\subset\cdots\subset (F'_1)^\perp\subset (F'_0)^\perp=\bC^{2n+1}).
\end{align*}
The flags $F,F'$ are transverse, in the sense that the intersection of any two of their constituent subspaces is transverse.

Let $I\subset[n]$ be a subset. Let $\lambda_1>\cdots>\lambda_s$ be the elements of $I$, which we may identify with the partition $\lambda=(\lambda_1,\ldots,\lambda_s)$. Let $\mu_{s+1}>\ldots>\mu_{n}$ be the non-negative integers obtained by removing the integers $n-\lambda_s,\ldots,n-\lambda_1$ from $n-1,\ldots,0$. For example, if $I=\{1,4,6\}\subset\{1,2,3,4,5,6,7\}$, then $(\mu_4,\mu_5,\mu_6,\mu_7)=(5,4,2,0)$.

\begin{definition}\label{def:schubert}
Define the \emph{Schubert variety} $\Sigma_I^F\subset\OG(n,2n+1)$ to be the closure of the subscheme $\Sigma_I^{F,\circ}\subset\OG(n,2n+1)$ of isotropic subspaces $\Lambda\subset\bC^{2n+1}$ satisfying:
\begin{align*}
    \dim(\Lambda\cap F_{n+1-\lambda_h})&=h,\text{ for }h=1,\ldots,s,\\
    \dim(\Lambda\cap F^\perp_{\mu_h})&=h,\text{ for }h=s+1,\ldots,n.
\end{align*}
\end{definition}
The Schubert variety $\Sigma_I^F$ is non-empty, reduced and irreducible of codimension 
\begin{equation*}
w(I):=\lambda_1+\cdots+\lambda_s
\end{equation*}
This may be seen, for example, by realizing $\Sigma_I^{F,\circ}$ as a double coset in $\OG(n,2n+1)$. In particular, $\Sigma_I^F$ defines a class $\sigma_I:=[\Sigma_I^F]\in H^{2w(I)}(\OG(n,2n+1))$. The classes of Schubert varieties freely generate $H^{*}(\OG(n,2n+1))$ \cite{bgg,borel}.

Similarly, if $I'\subset[n]$ is a subset, to which we associate the partitions $\lambda',\mu'$ as above, then we may define the Schubert variety $\Sigma_{I'}^{F'}\subset\OG(n,2n+1)$, of class $\sigma_{I'}=[\Sigma_{I'}^{F'}]\in H^{2w(I')}(\OG(n,2n+1))$.

\begin{definition}\label{def:richardson}
Let $I,I'\subset[n]$ be subsets. Define the \emph{Richardson variety} $\Sigma_{I,I'}:=\Sigma_I^F\cap \Sigma_{I'}^{F'}\subset\OG(n,2n+1)$.
\end{definition}

A standard incidence correspondence argument varying the flags $F,F'$ shows that $\Sigma_{I,I'}$ is reduced and irreducible of codimension $w(I)+w(I')$, the expected, and that a general point of the intersection in fact lies in the open locus $\Sigma_I^{F,\circ}\cap \Sigma_{I'}^{F',\circ}$. In particular, we have
\begin{equation*}
    [\Sigma_{I,I'}]=\sigma_I\sigma_{I'}\in H^{2(w(I)+w(I'))}(\OG(n,2n+1)).
\end{equation*}
Because our Richardson varieties will always be defined with respect to the flags $F,F'$, we drop the flags from the notation.

\subsection{Delta-matroids and polytopes}\label{sec:delta}

In this section, we discuss the (realizable) delta-matroid $\cD(\Lambda)$ associated to a point $\Lambda\in\OG(n,2n+1)$, and various auxiliary objects. See \cite{dh,bouchet,ck,bgw} for foundations of the theory of delta-matroids and \cite{efls,els,larson} for more recent work.

\begin{definition}\label{def:admissible}
    A subset $S\subset \wtn$ is \emph{admissible} if, for all $q\in[n]$, we have $\{q,\barq\}\not\subset S$. An admissible subset $S$ is \emph{maximal} if $|S|=n$.

    The collection of admissible subsets $S\subset\wtn$ is denoted $\Adm_n$, and the collection of maximal admissible subsets is denoted $\mAdm_n$.
\end{definition}

\begin{definition}\label{def:rank}
Let $\Lambda\in\OG(n,2n+1)$ be an isotropic subspace, and let $S\subset \wtn$ be an admissible subset. We define $\rk_\Lambda(S)\in[0,|S|]$ by the rank of the linear map
\begin{equation*}
\Lambda\hookrightarrow\bC^{2n+1}\to \bC^{2n+1}/H_S,
\end{equation*}
where we recall from \S\ref{sec:OG} that $H_S$ is the subspace of $\bC^{2n+1}$ spanned by the vectors $e_q$, $q\notin S$.

We also define
\begin{equation*}
    g_\Lambda(S):=\rk_\Lambda(S)-|S\cap\overline{[n]}|.
\end{equation*}
\end{definition}

Concretely, if $A$ is any $n\times (2n+1)$ matrix whose rows form a basis of $\Lambda$, then $\rk_\Lambda(S)$ is the rank of the submatrix $A_S$ (see \S\ref{sec:conventions}).

\begin{definition}\label{def:delta_matroid}
    Let $\Lambda\in\OG(n,2n+1)$ be an isotropic subspace. The \emph{delta-matroid} $\cD(\Lambda)\subset\mAdm_n$ associated to $\Lambda$ is the collection of maximal admissible  $S\subset \wtn$ for which $\rk_\Lambda(S)=n$.
\end{definition}

In fact, the data of $\cD(\Lambda)$ determines the value of $\rk_\Lambda(S)$, and therefore the value of $g_\Lambda(S)$, for \emph{any} admissible subset $S$, not necessarily maximal, see Corollary \ref{cor:rank_in_terms_of_M} of the Appendix. We deal only with realizable delta-matroids in this paper, that is, those that arise as $\cD(\Lambda)$ for some $\Lambda\in\OG(n,2n+1)$; the data of $\Lambda$ will always be present.

We now define the polytope $\cP(\Lambda)\subset\bR^n$ associated to $\Lambda$. Let $\mathbb{R}^n$ be an $n$-dimensional real vector space, with standard basis $f_1,\ldots,f_n$. 

\begin{definition}\label{def:x(S)}
Given a point $x=x_1f_1+\cdots+x_nf_n\in\mathbb{R}^n$ and an admissible subset $S\in\Adm_n$, define 
\begin{equation*}
    x(S):=\sum_{q=1}^{n}\epsilon_qx_q\in\mathbb{R}
\end{equation*}
where
\begin{equation*}
    \epsilon_q:=
    \begin{cases}
        1&\text{ if }q\in S\\
        -1&\text{ if }\barq\in S\\
        0&\text{ otherwise}
    \end{cases}
\end{equation*}
\end{definition}

\begin{definition}\label{def:polytope}
    We define $\cP(\Lambda)\subset\bR^n$ to be the convex polytope cut out by the inequalities
    \begin{equation*}
        x(S)\le g_\Lambda(S)
    \end{equation*}
    for all $S\in\Adm_n$.
\end{definition}
Taking $S$ to be each of the singleton sets $\{q\},\{\barq\}$ yields immediately that $\cP(\Lambda)\subset[0,1]^n$. For a general point $\Lambda\in\OG(n,2n+1)$, the polytope $\cP(\Lambda)$ is the full hypercube $[0,1]^n$. The polytope $\cP(\Lambda)$ is the base polytope of the delta-matroid $\cD(\Lambda)$, which is usually defined to be the convex hull of the $0-1$ vectors
\begin{equation*}
    \chi(S'):=\sum_{q\in [n]\cap S'}f_q\in\mathbb{R}^n.
\end{equation*}
ranging over maximal admissible $S'\in\cD(\Lambda)$. A proof of the equivalence of these definitions is given in \S\ref{sec:polytopes} of the appendix. 

Recall from \S\ref{sec:OG} that $\Lambda\in\OG(n,2n+1)$ gives rise to a torus orbit closure $Z_\Lambda\subset\OG(n,2n+1)$. The image of $Z_\Lambda$ under the \emph{moment map}
\begin{equation*}
    \mu:Z_\Lambda\hookrightarrow \bP(\wedge^n\bC^{2n+1})\to\bR^{n}
\end{equation*}
is equal to twice-dilated polytope
\begin{equation*}
\widehat{\cP(\Lambda)}:=2\cP(\Lambda)-(1,\ldots,1),
\end{equation*}
see \cite[Proposition 6.2]{efls} and \cite[\S 5.2, Proposition 1]{gs}. The inclusion $Z_\Lambda\hookrightarrow \bP(\wedge^n\bC^{2n+1})$ is the Pl\"{u}cker embedding, and the map $\bP(\wedge^n\bC^{2n+1})\to\bR^{n}$ is defined by
\begin{equation*}
    \sum_{S} x_Se_S\mapsto \sum_{S}\frac{|x_S|^2}{\sum_{S}|x_S|^2}f_S,
\end{equation*}
where 
\begin{equation*}
    f_S:=\sum_{q\in S\cap[n]}f_q-\sum_{q\in \overline{S}\cap [n]}f_{q}\in\bR^n,
\end{equation*}
and the sums are over \emph{all} (not necessarily admissible) subsets $S\subset\wtn$ of cardinality $n$. Moreover, the dimension of $\widehat{\cP(\Lambda)}$ (and thus $\cP(\Lambda)$) as a polytope in $\bR^n$ equals the dimension of $Z_\Lambda$ as a subvariety of $\OG(n,2n+1)$ \cite[\S 3.2, Theorem 1]{gs}.

\subsection{Degenerations of torus orbits}\label{sec:prelim_degen}

Let $U\subset\Spec(\bC[t])=\bA^1_\bC$ be a Zariski open subset containing the point $t=0$. (One can just as well take $U$ to be an analytic open disk, or the spectrum of  $\bC[t]_{(t)}$ or $\bC[[t]]$.)

We will be interested in degenerations of torus orbit closures in $\OG(n,2n+1)$, by which we mean proper flat families
\begin{equation*}
    \pi:\cZ\subset \OG(n,2n+1)\times U\to U,
\end{equation*}
which are fiberwise $T$-invariant, and whose (geometric) fibers away from $t=0$ are given by torus orbit closures $Z_{\Lambda_t}$. Equivalently, $\pi$ is a 1-parameter family in the Chow quotient $\OG(n,2n+1)//T$ in the sense of \cite{ksz}. (More precisely, $\pi$ is a 1-parameter family in $\bP(\wedge^n\bC^{2n+1})//T$ lying in the image of $\OG(n,2n+1)$.)

As an algebraic cycle, the special fiber $\pi^{-1}(0)$ is given by a positive linear combination of irreducible components
\begin{equation*}
    c_1Z_{1}+\cdots+c_mZ_{m}.
\end{equation*}
of the same dimension as $Z_{\Lambda_t}$. Each of the components $Z_k$ is itself a torus orbit closure $Z_{\Lambda_k}$ for some $\Lambda_k\in\OG(n,2n+1)$, and the polytopes $\cP(\Lambda_k)$ form a polyhedral decomposition of $\cP(\Lambda)$ \cite[Proposition 3.6]{ksz}, see also \cite[Proposition 1.2.11]{kapranov} for the case of ordinary Grassmannians. 

Furthermore, the coefficients $c_k$ may be computed in terms of the moment polytopes as follows, see \cite[(1.3)]{ksz} and the proof of \cite[Proposition 1.2.15]{kapranov}. Let $\Xi$ (resp $\Xi_k$) be the lattice generated by differences of vertices of $\cP$ (resp. $\cP_k$). Then, $c_k=[\Xi:\Xi_k]$.

\begin{corollary}\label{cor:moment_map_union}
    Let $\pi$ be a degeneration of torus orbits as above, whose general fiber is the torus orbit closure $Z_{\Lambda_t}$. Suppose that $Z_{\Lambda_1},\ldots,Z_{\Lambda_{m'}}$ are distinct torus orbit closures contained in the special fiber $\pi^{-1}(0)$, each of the same dimension of $Z_{\Lambda_t}$. Suppose further that
    \begin{equation*}
        \cP(\Lambda_t)=\bigcup_{k=1}^{m'}\cP(\Lambda_{i}).
    \end{equation*}
    Then, the irreducible components of the special fiber $\pi^{-1}(0)$ are precisely $Z_{\Lambda_1},\ldots,Z_{\Lambda_{m'}}$. In particular, we have
    \begin{equation*}
        [Z_{\Lambda_t}]=\sum_{k=1}^{m'}c_k[Z_{\Lambda_k}]\in H^{2\dim(\Lambda_t)}(\OG(n,2n+1)).
    \end{equation*}
\end{corollary}

\begin{proof}
    If $\dim(Z_{\Lambda_i})=\dim(Z_{\Lambda_t})$, then by flatness, $Z_{\Lambda_i}$ must be an irreducible component of the special fiber. If there were any other components of the special fiber, then their (dilated) moment map images would be non-empty sub-polytopes of $\cP(\Lambda_t)$ of the same dimension, and along with the $\cP(\Lambda_{i})$, would form a polyhedral decomposition of $\cP(\Lambda_t)$. However, this is impossible by the assumption on the union of the $\cP(\Lambda_i)$. The formula in cohomology follows, as the class remains constant in a flat family of algebraic cycles.
\end{proof}

\section{Charts for Richardson varieties}\label{sec:charts}

Our degeneration of the general torus orbit of in $\OG(n,2n+1)$ passes through torus orbits of subspaces $\Lambda\in\Sigma_{I,I'}$ lying in successively more special Richardson varieties, associated to \emph{allowed pairs} $(I,I')$. In order to make these degenerations explicit, we will need affine coordinates on the Richardson varieties $\Sigma_{I,I'}$, given in the end of this section by Corollary \ref{cor:chart}.

\subsection{Allowed pairs}\label{sec:allowed_pairs}



\begin{definition}\label{def:allowed_pairs}
    Let $I, I' \subset [n-1]$ be subsets. We say that pair $(I, I')$ is an \emph{allowed pair} if $I\cap I'=\emptyset$, and in addition one of the following holds:
        \begin{enumerate}
            \item[(i)] $I\cup I'=[n-1]$,
            \item[(ii)] $(I,I)=(\emptyset,\emptyset)$,
            \item[(iii)] $(I,I')=(\emptyset,\{j\})$ for some $j\in[n-2]$,
            \item[(iv)] for some $k\in[2,n-1]$, we have $k\in I'$ and $I\cup I'=[k,n-1]$, or
            \item[(v)] for some $j,k\in[n-1]$ with $1\le j\le k-2$, we have $j,k\in I'$ and $I\cup I'=\{j\}\cup[k,n-1]$.
        \end{enumerate}
\end{definition}

We say that the allowed pair $(I,I')$ is \emph{saturated} if $I\cup I'=[n-1]$ (case (i)). It is clear in cases (iii)-(v) that the integers $j,k$ are intrinsic to the pair of subsets $(I,I')$: for example, in case (v), $j$ is the smallest element of $I\cup I'$, and $k$ is the second-smallest, both lying necessarily in $I'$.

For indexing purposes, it will be convenient to view (ii)-(iv) as degenerate cases of (v). Cases (iii) and (iv) can be taken to be the degenerate cases of (v) where $k=n$ and $j=0$, respectively, and case (ii) can be taken to be the further degenerate case where $(j,k)=(0,n)$. In the degenerate cases, $I,I'$ are still taken to be subsets of $[1,n-1]$, so requirements that $0,n\in I'$ are ignored. For example, taking $k=n$ in case (v) reads: for some $j$ with $1\le j\le n-2$, we have $j,n\in I'$ and $I\cup I'\in \{j\}\cup [n,n-1]=\{j\}$. We interpret this to mean that $j\in I'$ and $I\cup I'=\{j\}$, which is precisely case (iii).

\begin{definition}\label{def:jk}
Let $(I,I')$ be an allowed pair. If $(I,I')$ is not saturated, then define the non-negative integers $j(I,I'),k(I,I')\le n$ to be the values of $j,k$ above. If $(I,I')$ is saturated, then we define $j(I,I')=0$ and $k(I,I')=1$.
\end{definition}

We have $0\le j(I,I')\le k(I,I')-2$, unless $(I,I')$ is saturated. Note that if $(I,I')$ is saturated, it is no longer the case that $k(I,I')=1$ necessarily lies in $I'$. We now organize the set of allowed pairs into a rooted binary tree $\cT_n$, see Definition \ref{def:T_n}.
\begin{definition}\label{def:tree_successors}
    Let $(I,I')$ be a non-saturated allowed pair, and write $j=j(I,I')$ and $k=k(I,I')$. Define the subsets $I_+, I'_+ \subset [n-1]$ by
\begin{align*}
    I_+ &= [j+1,k-1]\cup I,\\
    I'_+ &= (I' \setminus \{j\}) \cup\{j+1\}.
\end{align*}

Define also
\begin{align*}
    \ell((I,I'))&=(I,I'_+),\\
    r((I,I'))&=(I_+,I').
\end{align*}
\end{definition}

By definition, the non-empty interval $[j+1,k-1]$ is disjoint from $I$ and $I'$, so $I_+$ strictly contains $I$ and remains disjoint from $I'$. If $j\ge1$, then also by definition, we have that $j$ is the minimum element of $I'$ and that $j+1$ does not lie in $I'$ or $I$. Thus, $I'_+$ increments this minimum element by 1 and remains disjoint from $I$. If $j=0$, then by definition $1\notin I,I'$, so $I'_+$ simply adds 1 to $I'$. By construction, $\ell((I,I'))$ and $r((I,I'))$ are also allowed pairs.

\begin{definition}\label{def:T_n}
    Let $\cT_n$ be the rooted binary tree constructed iteratively as follows. The root vertex is labeled by the allowed pair $(\emptyset,\emptyset)$. Every vertex of $\cT_n$ labeled by a non-saturated allowed pair $(I,I')$ gives rise to left and right child vertices labeled by the allowed pairs $\ell((I,I'))$ and $r((I,I'))$, respectively. The $2^{n-1}$ leaves of $\cT_n$ are the saturated allowed pairs.
\end{definition}

\begin{example}
The tree $\cT_4$ is shown below.
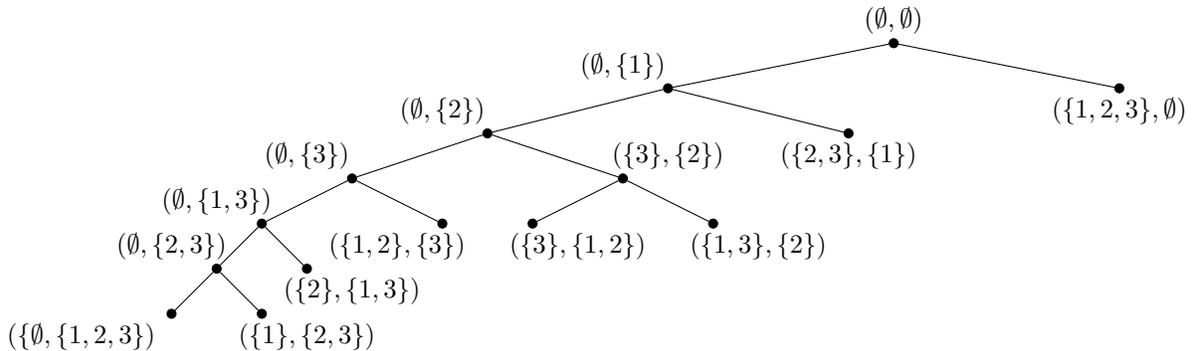
\begin{center}
    \begin{tikzpicture}[scale=0.6,font=\footnotesize]
        \tikzset{
        solid node/.style={circle,draw,inner sep=1.2,fill=black},
        hollow node/.style={circle,draw,inner sep=1.2}
        }
        \node[solid node](1) at (0,0) {};
        \node at (0,0.5) {$(\emptyset,\emptyset)$};
        \node[solid node](2) at (-5,-1) {};
        \node at (-6,-0.5) {$(\emptyset,\{1\})$};
        \node[solid node](3) at (5,-1) {};
        \node at (5,-1.5) {$(\{1,2,3\},\emptyset)$};
        \node[solid node](4) at (-9,-2) {};
        \node at (-10,-1.5) {$(\emptyset,\{2\})$};
        \node[solid node](5) at (-1,-2) {};
        \node at (-1,-2.5) {$(\{2,3\},\{1\})$};
        \node[solid node](6) at (-12,-3) {};
        \node at (-13,-2.5) {$(\emptyset,\{3\})$};
        \node[solid node](7) at (-6,-3) {};
        \node at (-5,-2.5) {$(\{3\},\{2\})$};
        \node[solid node](8) at (-14,-4) {};
        \node at (-15,-3.5) {$(\emptyset,\{1,3\})$};
        \node[solid node](9) at (-10,-4) {};
        \node at (-11,-4.5) {$(\{1,2\},\{3\})$};
        \node[solid node](10) at (-8,-4) {};
        \node at (-7,-4.5) {$(\{3\},\{1,2\})$};
        \node[solid node](11) at (-4,-4) {};
        \node at (-3,-4.5) {$(\{1,3\},\{2\})$};
        \node[solid node](12) at (-15,-5) {};
        \node at (-16,-4.5) {$(\emptyset,\{2,3\})$};
        \node[solid node](13) at (-13,-5) {};
        \node at (-12,-5.5) {$(\{2\},\{1,3\})$};
        \node[solid node](14) at (-14,-6) {};
        \node at (-13,-6.5) {$(\{1\},\{2,3\})$};
        \node[solid node](15) at (-16,-6) {};
        \node at (-18,-6.5) {$(\{\emptyset,\{1,2,3\})$};
	\path (1) edge (2);
	\path (1) edge (3);
        \path (2) edge (4);
        \path (2) edge (5);
        \path (4) edge (6);
        \path (4) edge (7);
	\path (6) edge (8);
	   \path (6) edge (9);
	 \path (7) edge (10);
	 \path (7) edge (11);    
	 \path (8) edge (12);
	 \path (8) edge (13);
	 \path (12) edge (14);
	 \path (12) edge (15);
        
    \end{tikzpicture}
\end{center}
\end{example}

Every allowed pair appears exactly once in $\cT_n$; we identify the set of vertices of $\cT_n$ with the set of allowed pairs. Indeed, any allowed $(I,I')$ can be obtained in a unique way by starting with $(\emptyset,\emptyset)$ and iteratively either incrementing the minimum element of $I'$ (or adding 1) or adding the ``gap'' $[j+1,k-1]$ to $I$. This corresponds to the unique downward path through $\cT_n$ from the root vertex to the vertex associated to $(I,I')$.

\begin{example}\label{eg:tree_path}

Take $n=9$, and let $(I,I')$ be the saturated allowed pair $(\{1,4,5,6,8\},\{2,3,7\})$. The unique path from the root of $\cT_9$ to $(I,I')$ is:
\begin{align*}
    &(\emptyset,\emptyset) \to (\emptyset,\{1\})\to (\emptyset,\{2\})\to \cdots\to (\emptyset,\{7\}) \to (\{8\},\{7\})\\
    \to& (\{8\},\{1,7\})\to (\{8\},\{2,7\})\to (\{8\},\{3,7\}) \to (\{4,5,6,8\},\{3,7\})  \\
    \to& (\{4,5,6,8\},\{1,3,7\}) \to (\{4,5,6,8\},\{2,3,7\}) \to (\{1,4,5,6,8\},\{2,3,7\}\}).
\end{align*}  
\end{example}

Roughly speaking, the elements of $I'$ must be set in order from largest to smallest by successively taking the left fork down $\cT_n$, and the gaps between consecutive elements of $I'$ must be filled into $I$ at the appropriate times by taking the right fork.

\subsection{Matrices for allowed pairs}\label{sec:matrices}
 
Let $(I,I')$ be an allowed pair, and write $j=j(I,I')$ and $k=k(I,I')$. In order to give local coordinates for $\Sigma_{I,I'}$, we first construct an $n\times (2n+1)$ matrix $M_{I,I'}$, each of whose entries is either 0 or one of the symbols $*,+,1$. Eventually, $M_{I,I'}$ will represent a general point of $\Sigma_{I,I'}$. 

Write $I = \{\lambda_1, \ldots, \lambda_s\}$, $I' = \{\lambda'_1, \ldots, \lambda'_{s'}\}$, where $\lambda_1>\cdots>\lambda_s$ and $\lambda'_1>\cdots>\lambda'_{s'}$. We also write $\mu_{s+1}>\cdots>\mu_n=0$ and $\mu'_{s'+1}>\cdots>\mu'_n=0$ as in \S\ref{sec:schubert}. The columns of $M_{I,I'}$ will be indexed (in order) by the set $\{1,\ldots,n,0,\barn,\ldots,\bar1\}$, and the rows are indexed by the set $[n]$. 

We now set the entries of $M_{I,I'}$, beginning with those that are 0.
\begin{enumerate}
    \item[(1)] If $1 \le p \le s'$, then in the $p$-th row of $M_{I,I'}$, place the symbol 0 in the leftmost $(n+\lambda'_p)$ entries.
    \item[(2)] If $s'+1\le p \le n$, then in the $p$-th row of $M_{I,I'}$, place the symbol 0 in the leftmost $\mu'_p$ entries.
\end{enumerate}
Dually:
\begin{enumerate}
    \item[(3)] If $n \ge p \ge n-s+1$, then in the $p$-th row of $M_{I,I'}$, place the symbol 0 in the rightmost $(n+\lambda_{n+1-p})$ entries.
    \item[(4)] If $n-s\ge p \ge 1$, then in the $p$-th row of $M_{I,I'}$, place the symbol 0 in the rightmost $\mu_{n+1-p}$ entries.
\end{enumerate}
We then fill the remaining entries of $M_{I,I'}$ with the symbols $*,+,1$, as follows:
\begin{enumerate}
    \item[(5)] In the rightmost unfilled entry of each row of $M_{I,I'}$, place the symbol 1.
    \item[(6)] In all remaining unfilled entries $a^p_{\barq}$ (where $\barq\in\overline{[n]}$) and $a^p_0$, place the symbol $+$.
    \item[(7)] In all remaining unfilled entries $a^p_q$ (where $q\in[n]$), place the symbol $*$.
\end{enumerate}

\begin{example}\label{eg:MII'}
    For $n = 8$, $I = \{4,6,7\}, I' = \{1,3,5\}$, we have
        \begin{equation*}
        M_{I,I'}=
        \begin{bmatrix}
            0 & 0 & 0 & 0 & 0 & 0 & 0 & 0 & 0 & 0 & 0 & 0 & 0 & + & + & + & 1 \\
            0 & 0 & 0 & 0 & 0 & 0 & 0 & 0 & 0 & 0 & 0 & + & + & 1 & 0 & 0 & 0 \\
            0 & 0 & 0 & 0 & 0 & 0 & 0 & 0 & 0 & + & + & 1 & 0 & 0 & 0 & 0 & 0 \\
            0 & 0 & 0 & 0 & 0 & 0 & * & * & + & + & 1 & 0 & 0 & 0 & 0 & 0 & 0 \\
            0 & 0 & 0 & 0 & * & * & * & * & + & 1 & 0 & 0 & 0 & 0 & 0 & 0 & 0 \\
            0 & 0 & * & * & 1 & 0 & 0 & 0 & 0 & 0 & 0 & 0 & 0 & 0 & 0 & 0 & 0 \\
            0 & * & 1 & 0 & 0 & 0 & 0 & 0 & 0 & 0 & 0 & 0 & 0 & 0 & 0 & 0 & 0 \\
            * & 1 & 0 & 0 & 0 & 0 & 0 & 0 & 0 & 0 & 0 & 0 & 0 & 0 & 0 & 0 & 0 \\
        \end{bmatrix}
    \end{equation*}
\end{example}

To obtain points of the Richardson variety $\Sigma_{I,I'}\subset\OG(n,2n+1)$, we replace the entries $*$ and $+$ in $M_{I,I'}$ with complex numbers, to obtain a new matrix $A$. (The name $M_{I,I'}$ is reserved for the matrix of symbols constructed above, whereas $A$ will denote a matrix whose entries are complex numbers.) An entry $a^p_q$ of the new matrix $A$ is called a ``$-$ entry,'' where $-\in\{0,1,+,*\}$, if the corresponding entry of $M_{I,I'}$ is $-$. When we say that an entry $a^p_q$ of $A$ is ``non-zero,'' we mean that $a^p_q\in\bC-\{0\}$. In particular, this is stronger than saying that $a^p_q$ is not a 0-entry.

\begin{definition}
    Let $\cM_{I,I'}$ be the parameter space of $n\times (2n+1)$ matrices $A$ obtained by replacing the symbols $+,*$ in $M_{I,I'}$ with complex numbers, with the property that the rows of $A$ are orthogonal.
\end{definition}

Write $b(I,I')$ and $d(I,I')$ for the numbers of $*$- and $+$-entries of $M_{I,I'}$, respectively. Then, $\cM_{I,I'}$ is the closed affine subscheme
    \begin{equation*}
    \cM_{I,I'}\subset \bC^{b(I,I')}\times \bC^{d(I,I')}
    \end{equation*}
cut out by the equations corresponding to orthogonality of the rows of $A$.

\begin{lemma}\label{lem:AII'}
Let $A\in\cM_{I,I'}$ be a matrix. Then, the rows of $A$ are linearly independent.
\end{lemma}

\begin{proof}
    By steps (3) and (4) of the construction of $M_{I,I'}$, the 1-entries of $A$ appear, from bottom to top, in columns
    \begin{equation*}
    (n+1)-\lambda_1,\ldots,(n+1)-\lambda_s,\overline{1+\mu_{s+1}},\ldots,\overline{1+\mu_n}=\overline{1},
    \end{equation*}
    which are all distinct. The entries to the right of the 1-entries are zero. The linear independence follows.
\end{proof}
    
In particular, we obtain a map $\ssp:\cM_{I,I'}\to \OG(n,2n+1)$ given by taking row span. The positions of the 0-entries in $M_{I,I'}$ guarantee further that the image of $\ssp$ lies in the Richardson variety $\Sigma_{I,I'}$. We will see that a general point of $\cM_{I,I'}$ in fact maps to the dense open locus $\Sigma_I^{F,\circ}\cap \Sigma_{I'}^{F',\circ}$ of $\Sigma_{I,I'}$.

We obtain two maps
\begin{equation*}
    \xymatrix{
    & \cM_{I,I'} \ar[dl]_{\pr} \ar[dr]^{\ssp} & \\
    \bC^{d(I,I')} & & \Sigma_{I,I'}
    }
\end{equation*}
where $\pr$ remembers the $+$-entries of a matrix $A$. We will show the following:
\begin{enumerate}
    \item[(1)] (Proposition \ref{prop:solve_for_*}) $\pr$ is birational.
    \item[(2)] (Corollary \ref{cor:matrix_injective}) On the locus of $\cM_{I,I'}$ where $\pr$ is an isomorphism, $\ssp$ is generically injective.
\end{enumerate}

Therefore, the rational map $\phi_{I,I'}:\bC^{d(I,I')}\dashrightarrow \Sigma_{I,I'}$ given by $\ssp\circ\pr^{-1}$ will be birational, for dimension reasons (Lemma \ref{lem:dII'}). This will give the needed chart for $\Sigma_{I,I'}$.

\subsection{Structure of $M_{I,I'}$}\label{sec:M_II'_structure}

We digress to give a useful alternate description of the matrix $M_{I,I'}$. First, define the \emph{inner box} $M^{\inbox}_{I,I'}$ of $M_{I,I'}$ to be the $k \times (2k+1)$ submatrix given by removing:
\begin{itemize}
    \item the leftmost $n-k$ and rightmost $n-k$ columns of $M_{I,I'}$, and
    \item the bottom $s$ and top $s'-1$ rows of $M_{I,I'}$, unless $j=0$, in which case the top $s'$ rows are removed.
\end{itemize}
We similarly write $A^{\inbox}$ for the corresponding submatrix of a matrix $A\in\cM_{I,I'}$. In particular, when $k=n$, we have $M^{\inbox}_{I,I'}=M_{I,I'}$. Equivalently, the inner box $M^{\inbox}_{I,I'}$ consists of the entries $m^p_q$ of $M_{I,I'}$, where 
\begin{equation*}
    p \in [|I' \setminus\{j\}| +1, n-|I|]\text{ and }q \in [n-k+1, n] \cup \overline{[n-k+1, n]} \cup \{0\}
\end{equation*}
Replacing $n$ with $k$ and re-indexing rows and columns appropriately, the submatrix $M^{\inbox}_{I,I'}$ is precisely the matrix $M_{\emptyset, \{j\}}$ if $j\neq 0$, and $M_{\emptyset, \emptyset}$ if $j=0$. An example of $M^{\inbox}_{I,I'}$ is depicted in Example \ref{eg:MII'_pieces}.

Explicitly, write $b^p_q$ for the entry in the $p$-th row and $q$-column of $M^{\inbox}_{I,I'}$, where we re-index the rows of $M^{\inbox}_{I,I'}$ by the set $[k]$ and the columns of $M^{\inbox}_{I,I'}$ by the set $\{1,\ldots,k,0,\bark,\ldots, \bar1\}$. In rows $p=2,\ldots,k$, the leftmost $k-p$ entries $b^p_1,\ldots,b^p_{k-p}$ and the rightmost $p-1$ entries $b^p_{\overline{p-1}},\ldots,b^p_{\bar1}$ of $M^{\inbox}_{I,I'}$ are all 0. If in addition $j\ge p$, then $b^p_{k-p+1}=0$ as well. All remaining entries $b^p_q$ in row $p$ are not 0. In row $p=1$, the leftmost $k-1$ entries $b^1_1,\ldots,b^1_{k-1}$ are 0. If $j=0$, then the remaining entries $b^1_q$ are not 0. Otherwise, the additional $j+1$ entries
\begin{equation*}
    b^1_k,b^1_0,b^1_{\bark},\ldots,b^1_{\overline{k-j+2}}
\end{equation*}
are 0, and the remaining entries $b^1_q$ are not 0.

Consider now the complement of $M^{\inbox}_{I,I'}$ in $M_{I,I'}$. Viewing the entries of $M_{I,I'}$ as points of a rectangular lattice, the remaining entries of $M_{I,I'}$ which are not 0-entries form two lattice paths $P_{\SW}$ and $P_{\NE}$. Example \ref{eg:MII'_pieces} also depicts these lattice paths. The lattice path $P_{\SW}$ travels from the bottom left entry of $M_{I,I'}$ to the bottom left entry of $M_{I,I'}^{\inbox}$.  If $k=n$, then these two endpoints are equal, so assume for what follows that $k<n$. The $s$ upward steps occur, from left to right, in columns 
\begin{equation*}
    n+1-\lambda_1<\cdots<n+1-\lambda_s.
\end{equation*}

Similarly, the lattice path $P_{\NE}$ travels from the top right entry of $M_{I,I'}^{\inbox}$ to the top right entry of $M_{I,I'}$. Assume further that $(I,I')$ is non-saturated, so $k>1$; we describe the saturated case later. If $j\neq0$, then $s'-1$ upward steps occur, from left to right, in the distinct columns
\begin{equation*}
    \overline{n+1-k}=\overline{n+1-\lambda'_{s'-1}},\ldots,\overline{n+1-\lambda'_1}.
\end{equation*}
If instead $j=0$, then there are $s'$ upward steps, in the columns
\begin{equation*}
    \overline{n+1-k}=\overline{n+1-\lambda'_{s'}},\ldots,\overline{n+1-\lambda'_1}.
\end{equation*}
In both cases, the left-most upward step is in column $\overline{n+1-k}$. Equivalently, the first step of $P_{\NE}$, starting from the upper right corner of $M_{I,I'}^{\inbox}$, is upward. 

We refer to the columns
\begin{equation*}
    n+1-\lambda_1,\ldots,n+1-\lambda_s,\overline{n+1-\lambda'_{s'}}\text{ or }\overline{n+1-\lambda'_{s'-1}},\ldots,\overline{n+1-\lambda'_1}.
\end{equation*}
in which the upward steps of $P_{\SW}$ and $P_{\NE}$ as the \emph{special columns} of $M_{I,I'}$. Note that these columns form an admissible subset of $[2,n]\cup\overline{[2,n]}$, that is, for all $j\in[2,n]$, an upward step of $P_{\SW}$ or $P_{\NE}$ appears in at most one of $j,\barj$.

\begin{example}\label{eg:MII'_pieces}
Continuing the example of \ref{eg:MII'}, where $n=8$ and $(I,I')=(\{4,6,7\},\{1,3,5\})$, the inner box $M_{I,I'}^{\inbox}$ is outlined, and the two lattice paths $P_{\SW}$ and $P_{\NE}$ are overlaid.
\begin{equation*}
\begin{tikzpicture}[row 8 column 17/.style=black]
\matrix(m) [matrix of math nodes,left delimiter={[},right delimiter={]}]
{
            0 & 0 & 0 & 0 & 0 & 0 & 0 & 0 & 0 & 0 & 0 & 0 & 0 & + & + & + & 1 \\
            0 & 0 & 0 & 0 & 0 & 0 & 0 & 0 & 0 & 0 & 0 & + & + & 1 & 0 & 0 & 0 \\
            0 & 0 & 0 & 0 & 0 & 0 & 0 & 0 & 0 & + & + & 1 & 0 & 0 & 0 & 0 & 0 \\
            0 & 0 & 0 & 0 & 0 & 0 & * & * & + & + & 1 & 0 & 0 & 0 & 0 & 0 & 0 \\
            0 & 0 & 0 & 0 & * & * & * & * & + & 1 & 0 & 0 & 0 & 0 & 0 & 0 & 0 \\
            0 & 0 & * & * & 1 & 0 & 0 & 0 & 0 & 0 & 0 & 0 & 0 & 0 & 0 & 0 & 0 \\
            0 & * & 1 & 0 & 0 & 0 & 0 & 0 & 0 & 0 & 0 & 0 & 0 & 0 & 0 & 0 & 0 \\
            * & 1 & 0 & 0 & 0 & 0 & 0 & 0 & 0 & 0 & 0 & 0 & 0 & 0 & 0 & 0 & 0 \\
  };

\draw[purple] (m-3-6.north west) -- (m-3-12.north east) -- (m-5-12.south east) -- (m-5-6.south west) -- (m-3-6.north west);
\draw[red, thick] (m-8-1.center) -- (m-8-2.center) -- (m-7-2.center) -- (m-7-3.center) -- (m-6-3.center) -- (m-6-5.center) -- (m-5-5.center) -- (m-5-6.center);
\draw[red, thick] (m-3-12.center) -- (m-2-12.center) -- (m-2-14.center) -- (m-1-14.center) -- (m-1-17.center);
 \end{tikzpicture}
   \end{equation*}

\end{example}

When $(I,I')$ is saturated, $\lambda_{s'}=\min(I')$ is not required to equal $k(I,I')=1$, so it may no longer be the case the first step of $P_{\NE}$ is upward. The inner box $M^{\inbox}_{I,I'}$ is simply the $1\times 3$ submatrix
\begin{equation*}
\begin{bmatrix}
    * & + & 1
\end{bmatrix}
\end{equation*}
if $1\in I'$, and 
\begin{equation*}
\begin{bmatrix}
    * & + & +
\end{bmatrix}
\end{equation*}
if $1\in I$. The two lattice paths $P_{\NE}$ and $P_{\SW}$ can be concatenated via the horizontal path through $A^{\inbox}_{I,I'}$ to form a single lattice path $P_{I,I'}$. The $n-1$ columns
\begin{equation*}
    n+1-\lambda_1,\ldots,n+1-\lambda_s,\overline{n+1-\lambda'_{s'}},\ldots,\overline{n+1-\lambda'_1}\in[2,n]\cup\overline{[2,n]}
\end{equation*}
in which the upward steps occur uniquely determine $P_{I,I'}$. Again, we refer to these colums as the \emph{special columns} of $M_{I,I'}$.

\begin{example}\label{eg:saturated_path}
Taking the saturated pair $(I,I')=(\{1,4,5,6,8\},\{2,3,7\})$ of Example \ref{eg:tree_path}, the matrix $M_{I,I'}$ is shown below, with the concatenated lattice path overlaid.
\begin{equation*}
\begin{tikzpicture}[row 9 column 19/.style=black]
\matrix(m) [matrix of math nodes,left delimiter={[},right delimiter={]}]
{
    0 & 0 & 0 & 0 & 0 & 0 & 0 & 0 & 0 & 0 & 0 & 0 & 0 & 0 & 0 & 0 & + & + & 1 \\
    0 & 0 & 0 & 0 & 0 & 0 & 0 & 0 & 0 & 0 & 0 & 0 & + & + & + & + & 1 & 0 & 0 \\
    0 & 0 & 0 & 0 & 0 & 0 & 0 & 0 & 0 & 0 & 0 & + & 1 & 0 & 0 & 0 & 0 & 0 & 0 \\
    0 & 0 & 0 & 0 & 0 & 0 & 0 & 0 & * & + & + & 1 & 0 & 0 & 0 & 0 & 0 & 0 & 0 \\
    0 & 0 & 0 & 0 & 0 & * & * & * & 1 & 0 & 0 & 0 & 0 & 0 & 0 & 0 & 0 & 0 & 0 \\
    0 & 0 & 0 & 0 & * & 1 & 0 & 0 & 0 & 0 & 0 & 0 & 0 & 0 & 0 & 0 & 0 & 0 & 0 \\
    0 & 0 & 0 & * & 1 & 0 & 0 & 0 & 0 & 0 & 0 & 0 & 0 & 0 & 0 & 0 & 0 & 0 & 0 \\
    0 & * & * & 1 & 0 & 0 & 0 & 0 & 0 & 0 & 0 & 0 & 0 & 0 & 0 & 0 & 0 & 0 & 0 \\
    * & 1 & 0 & 0 & 0 & 0 & 0 & 0 & 0 & 0 & 0 & 0 & 0 & 0 & 0 & 0 & 0 & 0 & 0 \\
  };

\draw[red, thick] (m-9-1.center) -- (m-9-2.center) -- (m-8-2.center) -- (m-8-4.center) -- (m-7-4.center) -- (m-7-5.center) -- (m-6-5.center) -- (m-6-6.center) -- (m-5-6.center) -- (m-5-9.center) -- (m-4-9.center) -- (m-4-12.center) -- (m-3-12.center) -- (m-3-13.center) -- (m-2-13.center) -- (m-2-17.center) -- (m-1-17.center) -- (m-1-19.center);

 \end{tikzpicture}
   \end{equation*}
\end{example}

\subsection{The map $\pr$}

We now return to our task of describing charts for $\Sigma_{I,I'}$. Recall that we have defined $\cM_{I,I'}$ to be the affine variety of matrices of shape $M_{I,I'}$, with orthogonal rows. There is a natural projection map
\begin{equation*}
    \pr:\cM_{I,I'}\to \bC^{d(I,I')}
\end{equation*}
remembering only the $+$-entries.

\begin{proposition}\label{prop:solve_for_*}
The map $\pr$ is birational. That is, there is a non-empty Zariski open subset $V_{I,I'}\subset \bC^{d(I,I')}$ over which $\pr$ is an isomorphism.
\end{proposition}

We will prove Proposition \ref{prop:solve_for_*} by showing that, for a general choice of $+$-entries of a matrix $A\in\cM_{I,I'}$, there is a unique choice of $*$-entries so that the rows of $A$ are orthogonal. This amounts to solving a system of linear equations in the $*$-entries whose coefficients are the chosen $+$-entries, and the open subset $V_{I,I'}$ will be simply the locus where the determinants of the coefficient matrices do not vanish.

We will first prove Proposition \ref{prop:solve_for_*} in two special cases (Propositions \ref{prop:solve_for_*_box} and \ref{prop:solve_for_*_saturated}), and then combine the arguments to prove the claim in general.

\begin{proposition}\label{prop:solve_for_*_box}
    Proposition \ref{prop:solve_for_*} holds when $I=\emptyset$ and either $I'=\emptyset$ or $I'=\{j\}$, where $j\in[1,n-2]$. 
\end{proposition}

\begin{proof}
As in Definition \ref{def:jk}, write $j=j(I,I')=0$ if $I'=\emptyset$, and otherwise let $j$ be the unique element of $I'$.

Write $v_p$ for the row vector in the $p$-th row of $A$, and write $v_p^{*}$ for the vector whose coordinates are the $*$-entries of $v_p$. We claim that the vectors $v_p^{*}$ can be uniquely solved for by pairing with the appropriate rows above. More precisely:
\begin{enumerate}
    \item[(1)] for $p=2,\ldots,j$, in increasing order, we solve for the $(p-1)$ $*$-entries of $v_p$ by considering the system of $p-1$ equations
    \begin{equation*}
        v_2\cdot v_p, v_3\cdot v_p,\ldots,v_p\cdot v_p=0.
    \end{equation*}
    \item[(2)] for $p=j+1,\ldots,n$, in increasing order, we solve for the $p$ $*$-entries of $v_p$ by considering the system of $p$ equations
    \begin{equation*}
        v_1\cdot v_p, v_2\cdot v_p,\ldots,v_p\cdot v_p=0.
    \end{equation*}
\end{enumerate}
Case (1) only applies when $j\ge2$. In case (1), the system of equations takes the form
\begin{equation*}
    A_p\cdot v_p^{*}=\gamma_p,
\end{equation*}
where $A_p$ is the submatrix of $A$ with row set $[2,p]$ and column set $\overline{[n-p+2,n]}$, and $\gamma_p\in\bC^{p-1}$ is a vector depending on $A$ and the $*$-entries solved for in the earlier rows. In case (2), $A_p$ is the submatrix of $A$ with row set $[1,p]$ and column set $\overline{[n-p+1,n]}$, and $\gamma_p\in\bC^{p}$ is again a vector depending on $A$ and the $*$-entries solved for in the earlier rows. The columns of $A_p$ are taken in the same order as they appear in $A$, so that column $\barn$ appears at the far left. See Example \ref{eg:innerbox} for a depicted example of the matrices $A_p$. The vector $v_p^{*}$ is now taken to be a column vector, with entries appearing from top to bottom with the \emph{right-most} $*$-entry $a^p_n$ on top.

We claim that, in both cases, the anti-diagonal of $A_p$ contains no 0-entries. This will imply that the determinant of $A_p$, viewed as a function of the $+$-entries of $A$, is not identically zero. Therefore, as long as the $+$-entries are chosen such that $\det(A_p)\neq0$ for all $p$, there will exist a unique solution $v_p^{*}$ for each $p$. We take $V_{I,I'}\subset (\bC-\{0\})^{d(I,I')}$ to be the non-empty Zariski open subset defined by the non-vanishing of the $\det(A_p)$.

By construction, any entry immediately to the southwest of a non-0-entry is also not a 0-entry. Therefore, to verify the claim, it suffices to check that the upper right entry of every $A_p$ is not a 0-entry. In increasing order of $p$, these upper right entries are:
\begin{equation*}
    a^2_{\barn},\ldots,a^2_{\overline{n-j+2}},a^1_{\overline{n-j}},\ldots,a^1_{\bar1}.
\end{equation*}
By construction, all of the above are $+$-entries, except $a^1_{\bar1}=1$, as needed.

Therefore, for a general choice of $+$-entries in $V_{I,I'}$, the vectors $v_p^{*}$ can be solved for uniquely, determining a unique matrix $A$. It suffices now to observe that the remaining dot products
\begin{equation*}
    v_1\cdot v_1,\ldots,v_1\cdot v_j
\end{equation*}
are automatically zero, for \emph{any} choice of $+$- and *-entries of $A$. Therefore, the matrix $A$ in fact lies in $\cM_{I,I'}$.
\end{proof}

\begin{example}\label{eg:innerbox}
The matrix $M_{I,I'}$ for $(I,I')=(\emptyset,\{3\})$. The positions of the submatrices $A_p$ from the proof of Proposiiton \ref{prop:solve_for_*_box}, $p=2,\ldots,6$, are outlined.
\begin{equation*}
\begin{tikzpicture}[row 6 column 13/.style=black]
\matrix(m) [matrix of math nodes,left delimiter={[},right delimiter={]}]
{
 0 & 0 & 0 & 0 & 0 & 0 & 0 & 0 & 0 & + & + & + & 1\\
 0 & 0 & 0 & 0 & 0 & * & + & + & + & + & + & 1 & 0\\
 0 & 0 & 0 & 0 & * & * & + & + & + & + & 1 & 0 & 0\\
 0 & 0 & * & * & * & * & + & + & + & 1 & 0 & 0 & 0\\
 0 & * & * & * & * & * & + & + & 1 & 0 & 0 & 0 & 0\\
 * & * & * & * & * & * & + & 1 & 0 & 0 & 0 & 0 & 0\\
  };

\draw[purple] (m-1-8.north west) -- (m-1-13.north east) -- (m-6-13.south east) -- (m-6-8.south west) -- (m-1-8.north west);
\draw[purple] (m-1-12.north east) -- (m-5-12.south east) -- (m-5-8.south west);
\draw[purple] (m-1-11.north east) -- (m-4-11.south east) -- (m-4-8.south west);
\draw[purple] (m-2-8.north west) -- (m-2-9.north east) -- (m-3-9.south east) -- (m-3-8.south west);
\draw[purple] (m-2-8.north east) -- (m-2-8.south east) -- (m-2-8.south west);
  
 \end{tikzpicture}
 \end{equation*}
\end{example}

\begin{proposition}\label{prop:solve_for_*_saturated}
    Proposition \ref{prop:solve_for_*} holds when $(I,I')$ is saturated. In fact, taking $V_{I,I'}\subset\bC^{d(I,I')}$ to be the locus where all of the $+$-entries are non-zero, $\pr$ is an isomorphism over $V_{I,I'}$. Furthermore, for any $A\in\pr^{-1}(V_{I,I'})$, all of the $*$-entries of $A$ are non-zero. 
\end{proposition}

\begin{proof}
As in the proof of Proposition \ref{prop:solve_for_*_box}, write $v_p$ for the row vector in the $p$-th row of $A$. The $*$-entries of $A$ only appear in rows $p=n-s,n-s+1,\ldots,n$. Write now $v_p^{*}$ for the vector whose coordinates are the $*$-entries of $v_p$, \emph{in addition to the 1-entry at the far right}.

Suppose first that $p\ge n-s+1$. The $*$-entries of $v_p$ appear in columns 
\begin{equation*}
    n-(\lambda_{n-p}-1),\ldots,n-\lambda_{n-p+1},
\end{equation*}
 and the 1-entry appears in column $(n+1)-\lambda_{n-p+1}$. Let $R_p$ be the set of rows which contain non-0-entries in at least one of the opposite columns
\begin{equation*}
    \overline{(n+1)-\lambda_{n-p}},\ldots,\overline{(n+1)-\lambda_{n-p+1}}.
\end{equation*}
Let $A_p$ be the matrix of size $(\lambda_{n-p}-\lambda_{n-p+1})\times (\lambda_{n-p}-\lambda_{n-p+1}+1)$ with row set $R_p$ and column set $\overline{[(n+1)-\lambda_{n-p},\ldots,(n+1)-\lambda_{n-p+1}]}$. See Example \ref{eg:saturated_submatrix} for a depicted example of the matrices $A_p$.

More precisely, the non-0-entries of $A_p$ form a subset of the lattice path $P_{\NE}$, alternating between rightward and upward steps, starting with a rightward step from southwest corner, and ending in a rightward step to the northeast corner. The upward steps are those in the columns in the interval $\overline{[n+2-\lambda_{n-p},n-\lambda_{n-p+1}]}$. In particular, if $\lambda_{n-p}-\lambda_{n-p+1}=1$, then this interval is empty. Accordingly, $R_p$ consists of a unique row, and no upward step occurs in columns $\overline{n+2-\lambda_{n-p}},\overline{n-\lambda_{n-p+1}}$. Each rightward step in $A_p$ starts from a $+$-entry and ends on a 1-entry.

Then, $v_p$ is orthogonal to every row of $A$ if and only if $A_p v_p^{*}=0$. Note that this equation does not depend on the values of the $*$-entries outside of the $p$-th row of $A$. By the structure of $A_p$ described above, there is a unique such vector $v_p^{*}$ with $A_p v_p^{*}=0$ for \emph{any} choice of non-zero $+$ entries of $A$, and all of its coordinates non-zero. Indeed, each $*$-entry in $v_p$ is determined by the entry immediately to the right by pairing with the appropriate row in $R_p$.

Similarly, when $p=n-s$, one can solve uniquely for the entries of $v_{n-s}^{*}$, all non-zero, from the equations
\begin{equation*}
    v_{n-s}\cdot v_{n-s}=0,\ldots,v_{n-s}\cdot v_{(n+1)-s-\lambda_s}=0.
\end{equation*}
The vectors $v_p^{*}$ for $p=n-s,\ldots,n$ uniquely determine the matrix $A$, and all $*$-entries are non-zero. Furthermore, $v_p$ is automatically orthogonal to any row outside $R_p$, so all of the rows of $A$ are orthogonal. This completes the proof.
\end{proof}

    \begin{example}\label{eg:saturated_submatrix}
Taking the saturated pair $(I,I')=(\{1,4,5,6,8\},\{2,3,7\})$ of Examples \ref{eg:tree_path} and \ref{eg:saturated_path}, the positions of the submatrices $A_p$ from the proof of Proposition \ref{prop:solve_for_*_saturated}, $p=5,6,7,8,9$, are outlined. For $p=6,7$, the matrices $A_p$ are the $1\times 2$ matrices supported in row 2. The $*$-entry in row 4 is determined by $v_4\cdot v_4=0$.
\begin{equation*}
\begin{tikzpicture}[row 9 column 19/.style=black]
\matrix(m) [matrix of math nodes,left delimiter={[},right delimiter={]}]
{
    0 & 0 & 0 & 0 & 0 & 0 & 0 & 0 & 0 & 0 & 0 & 0 & 0 & 0 & 0 & 0 & + & + & 1 \\
    0 & 0 & 0 & 0 & 0 & 0 & 0 & 0 & 0 & 0 & 0 & 0 & + & + & + & + & 1 & 0 & 0 \\
    0 & 0 & 0 & 0 & 0 & 0 & 0 & 0 & 0 & 0 & 0 & + & 1 & 0 & 0 & 0 & 0 & 0 & 0 \\
    0 & 0 & 0 & 0 & 0 & 0 & 0 & 0 & * & + & + & 1 & 0 & 0 & 0 & 0 & 0 & 0 & 0 \\
    0 & 0 & 0 & 0 & 0 & * & * & * & 1 & 0 & 0 & 0 & 0 & 0 & 0 & 0 & 0 & 0 & 0 \\
    0 & 0 & 0 & 0 & * & 1 & 0 & 0 & 0 & 0 & 0 & 0 & 0 & 0 & 0 & 0 & 0 & 0 & 0 \\
    0 & 0 & 0 & * & 1 & 0 & 0 & 0 & 0 & 0 & 0 & 0 & 0 & 0 & 0 & 0 & 0 & 0 & 0 \\
    0 & * & * & 1 & 0 & 0 & 0 & 0 & 0 & 0 & 0 & 0 & 0 & 0 & 0 & 0 & 0 & 0 & 0 \\
    * & 1 & 0 & 0 & 0 & 0 & 0 & 0 & 0 & 0 & 0 & 0 & 0 & 0 & 0 & 0 & 0 & 0 & 0 \\
  };

\draw[purple] (m-2-11.north west) -- (m-2-14.north east) -- (m-4-14.south east) -- (m-4-11.south west) -- (m-2-11.north west);

 \draw[purple] (m-2-14.north west) -- (m-2-15.north east) -- (m-2-15.south east) -- (m-2-14.south west) -- (m-2-14.north west);
 
  \draw[purple] (m-2-15.north west) -- (m-2-16.north east) -- (m-2-16.south east) -- (m-2-15.south west) -- (m-2-15.north west);

\draw[purple] (m-1-16.north west) -- (m-1-18.north east) -- (m-2-18.south east) -- (m-2-16.south west) -- (m-1-16.north west);

  \draw[purple] (m-1-18.north west) -- (m-1-19.north east) -- (m-1-19.south east) -- (m-1-18.south west) -- (m-1-18.north west);

 \end{tikzpicture}
   \end{equation*}

\end{example}

\begin{proof}[Proof of Proposition \ref{prop:solve_for_*}]
Consider a general choice of $+$-entries for a matrix $A\in\cM_{I,I'}$. Repeating the argument of Proposition \ref{prop:solve_for_*_box}, the $*$-entries inside $A^{\inbox}$ may be solved for uniquely. The remaining $*$-entries of $A$, all lying in the lattice path $P_{\SW}$, may be solved for uniquely by repeating the argument of Proposition \ref{prop:solve_for_*_saturated}, pairing the rows of $A$ containing points of $P_{\SW}$ with rows lying above $A^{\inbox}$. The second step includes in particular the $*$-entries in row $n-s$ (the bottom row of $A^{\inbox}$) but outside the inner box.
\end{proof}

\subsection{The map $\ssp$}

\begin{proposition}\label{prop:*_nonzero}
Let $(I,I')$ be an allowed pair. For some non-empty Zariski open subset $W_{I,I'}\subset V_{I,I'}\subset\bC^{d(I,I')}$ of $+$-entries, the following holds. For any $a\in W_{I,I'}$, write $A=\pr^{-1}(a)\in\cM_{I,I'}$. Then, the left-most non-0-entry in each row of $A$ is not equal to 0.
\end{proposition}

We will use a similar strategy as in the proof for Proposition \ref{prop:solve_for_*}. By the last statement of Proposition \ref{prop:solve_for_*_saturated}, it is essentially sufficient to consider the inner box of $A$, so we consider the case  $(I,I')=(\emptyset,\emptyset)$ or $(\emptyset,\{j\})$, where $j\in[n-2]$. Then, if the left-most non-0-entry in the $p$-th row of $A$ equals zero, we will see that a particular square sub-matrix of $A$, supported in the columns $0,\overline{[n]}$, must have zero determinant. The non-empty Zariski open subset $W_{I,I'}$ will be the locus where the determinants of these submatrices of $A$ are not $0$.

\begin{proposition}\label{prop:*_nonzero_box}
Proposition \ref{prop:*_nonzero} holds in the cases $(I,I')=(\emptyset,\emptyset)$ and $(I,I')=(\emptyset,\{j\})$, as in Proposition \ref{prop:solve_for_*_box}.
\end{proposition}

\begin{proof}
As before, let $j = j(I,I') = 0$ if $I' = \emptyset$. Otherwise, let $j$ be the unique element of $I'$. For any $p$, denote the $p$-th row of $A$ by $v_p$. When $j \ge 1$, the leftmost entry of the $v_1$ is a $+$-entry, so there is nothing to check. Otherwise, we prove by strong induction on $p$, starting from $p=n$ and decreasing the value of $p$, that the left-most $*$-entry of $v_p$ is non-zero.

Suppose first that $p\ge j+1$. By construction, the leftmost $\mu'_p=n-p$ entries of the $v_p$ are $0$-entries, so the leftmost non-$0$-entry, which is a $*$-entry, appears in column $n-p+1$. Suppose, for the sake of contradiction, that this entry is equal to $0$. This means that
    \begin{equation*}
        \dim(\Lambda \cap (F'_{n-p+1})^\perp) \ge p,
    \end{equation*}
    because now the first $p$ rows have $0$ in column $1, \ldots, n-p+1$. We claim that
    \begin{equation*}
        \dim(\Lambda \cap F'_{n-p+1}) \ge 1.
    \end{equation*}
Indeed, if $\Lambda \cap F'_{n-p+1}=0$, then $\Lambda \cap (F'_{n-p+1})^\perp$ may be viewed as an isotropic subspace of $(F'_{n-p+1})^\perp/F'_{n-p+1}$, which inherits a non-degenerate symmetric form from $\bC^{2n+1}$. However, we have $\dim((F'_{n-p+1})^\perp/F'_{n-p+1})=2p-1$, so we cannot have $\dim(\Lambda \cap (F'_{n-p+1})^\perp) \ge p$.

Let $\widehat{A}_p$ be the $n \times (n+p)$ submatrix of $A$ obtained by removing columns $\overline{n-p+1}, \ldots, \overline{1}$ from $A$. Because $\dim(\Lambda \cap F'_{n-p+1}) \ge 1$, the rank of $\widehat{A}_p$ is at most $n-1$. 

Write $\widehat{v}_1,\ldots,\widehat{v}_n$ for the rows of $\widehat{A}_p$. By the inductive hypothesis, the left-most non-$0$-entries of $\widehat{v}_n, \ldots, \widehat{v}_{p+1}$ are not zero, and appear in different columns, so $\widehat{v}_n, \ldots,\widehat{v}_{p+1}$ are linearly independent. On the other hand, by construction, the left-most $\mu'_{p}=n-p$ entries of row $p, \ldots, 1$ are all zero, so 
\begin{equation*}
    \langle \widehat{v}_n, \ldots, \widehat{v}_{p+1} \rangle \cap \langle \widehat{v}_{p}, \ldots, \widehat{v}_1 \rangle = 0.
\end{equation*}
Then, 
   \begin{align*}
       \dim(\langle \widehat{v}_n, \ldots,\widehat{v}_1 \rangle) &= \dim(\langle \widehat{v}_n, \ldots,\widehat{v}_{p+1} \rangle)+\dim(\langle \widehat{v}_{p}, \ldots,\widehat{v}_{1} \rangle)\\
       &= (n-p) + \dim(\langle \widehat{v}_{p}, \ldots,\widehat{v}_{1} \rangle).
   \end{align*}
    Thus, the vectors $\widehat{v}_{p}, \ldots,\widehat{v}_{1}$ must be linearly dependent. 
    
    On the other hand, let $\widetilde{A}_p$ be the further submatrix of $\widehat{A}_p$ with row set $[1,p]$ and column set $[0]\cup \overline{[n-p+2,n]}$. (Note that $\widehat{A}_p$ is obtained from the matrix $A_p$ from the proof of Proposition \ref{prop:solve_for_*_box} by shifting entries one unit to the left.) Then, we claim that the determinant of $\widehat{A}_p$ is not identically zero as a function of the $+$-entries of $A$, hence non-zero for a general choice of $+$-entries. Indeed, the anti-diagonal of $\widehat{A}_p$ has no $0$-entries, because the upper right entry $a^1_{\widehat{n-p+2}}$ is a non-$0$-entry when $p\ge j+1$. This contradicts the linear dependence of $\widehat{v}_{p}, \ldots,\widehat{v}_{1}$, and hence the assumption that the left-most $*$-entry of $v_p$ is zero.

Suppose next that $p\le j$. The argument is similar to before: now, if the left-most $*$-entry of $v_p$ is zero, then 
    \begin{equation*}
        \dim(\Lambda \cap (F'_{n-p+2})^\perp) \ge p,
    \end{equation*}
    which would imply that 
    \begin{equation*}
        \dim(\Lambda \cap F'_{n-p+2}) \ge 2.
    \end{equation*}
Note that $v_1\in \Lambda \cap F'_{n-p+2}$. On the other hand, arguing as before (and using the inductive hypothesis), this would imply that the $(p-1)\times (p-1)$ matrix $\widehat{A}_p$ with row set $[2,p]$ and column set $[0]\cup \overline{[n-p+3,n]}$ would have dependent rows. (Once more, $\widehat{A}_p$ is obtained from the matrix $A_p$ from the proof of Proposition \ref{prop:solve_for_*_box} by shifting entries one unit to the left.) This is again an impossibility for a general choice of $+$-entries of $A$, because the anti-diagonal entries of $\widehat{A}_p$ are not 0-entries.   
\end{proof}

Note that the last statement of Proposition \ref{prop:solve_for_*_saturated} gives Proposition \ref{prop:*_nonzero} when $(I,I')$ is saturated, by taking $W_{I,I'}=V_{I,I'}$. We now combine this with Proposition \ref{prop:*_nonzero_box} to prove Proposition \ref{prop:*_nonzero} in general.

\begin{proof}[Proof of Proposition \ref{prop:*_nonzero}]
    For a general choice of $+$-entries in $A$, the argument of Proposition \ref{prop:*_nonzero_box} shows that the left-most entry of every row of $A^{\inbox}$ is non-zero. In particular, the entry in the southwest corner of $A^{\inbox}$ is non-zero. (Note that this entry is \emph{not} in general the leftmost entry in its row of $A$.) Solving for the remaining $*$-entries along $P_{\SW}$, as in Proposition \ref{prop:solve_for_*_saturated} (and the proof of Proposition \ref{prop:solve_for_*}), shows further that \emph{all} remaining $*$-entries are non-zero. The Proposition follows.
\end{proof}

For each $p=1,\ldots,n$, let $\alpha'_p,\alpha_p$ be the integers such that the $p$-th row of $M_{I,I'}$ has its leftmost $\alpha'_p$ entries equal to zero, and its rightmost $\alpha_p$ entries equal to zero. Proposition \ref{prop:*_nonzero} implies that, for any $A\in\pr^{-1}(W_{I,I'})$, the $(\alpha'_p+1)$-st entry from the left of the $p$-th row is non-zero, as is the the $(\alpha_p+1)$-st entry from the right. In particular, if $\Lambda$ is the row span of $A$, we have
\begin{align*}
    \dim(\Lambda\cap F_{n+1-\lambda_h})&=h,\text{ for }h=1,\ldots,s,\\
    \dim(\Lambda\cap F^\perp_{\mu_h})&=h,\text{ for }h=s+1,\ldots,n.
\end{align*}
and 
\begin{align*}
    \dim(\Lambda\cap F'_{n+1-\lambda'_h})&=h,\text{ for }h=1,\ldots,s',\\
    \dim(\Lambda\cap (F'_{\mu'_h})^\perp)&=h,\text{ for }h=s'+1,\ldots,n,
\end{align*}
by construction of $M_{I,I'}$. Therefore, identifying $W_{I,I'}$ with its pre-image in $\cM_{I,I'}$, we obtain a map
\begin{equation*}
\ssp: W_{I,I'}\to \Sigma^{F,\circ}_I \cap \Sigma^{F',\circ}_{I'} \subset \Sigma_{I.I'}
\end{equation*}
given by taking row span.

\begin{corollary}\label{cor:matrix_injective}
The row span map $\ssp: W_{I,I'}\to \Sigma_{I.I'}$ is injective.
\end{corollary}

\begin{proof}
    For any $\Lambda\in\Sigma_{I,I'}$ in the image of $\ssp$, we have
    \begin{equation*}
        \dim(\Lambda\cap F'_{(2n+1)-\alpha'_p}\cap F_{(2n+1)-\alpha_p})=1.
    \end{equation*}
    Here, we write $F_m$ for the $m$-dimensional flag element of $F$, that is,
    \begin{equation*}
        F_m=
        \begin{cases}
            F_m&\text{ if }m\le n\\
            F_{2n+1-m}^\perp&\text{ if }m\ge n+1
        \end{cases}
    \end{equation*}
    and similarly for $F'_m$. Indeed, if $\Lambda=\ssp(A)$, then $\Lambda\cap F'_{(2n+1)-\alpha'_p}$ is the row span of the first $p$ rows of $A$, by Proposition \ref{prop:*_nonzero}. Similarly, $\Lambda\cap F_{(2n+1)-\alpha_p}$ is the row span of the last $n-p+1$ rows.

    It follows that any matrix mapping to $\Lambda$ under $\ssp$ must have, as its $p$-th row, the unique vector lying in $\Lambda\cap F'_{(2n+1)-\alpha'_p}\cap F_{(2n+1)-\alpha_p}$ whose rightmost non-zero coordinate equals 1. There is a unique such matrix, so the Corollary is proven.
\end{proof}

In particular, we obtain a rational map
\begin{equation*}
    \phi_{I,I'}:\bC^{d(I,I')}\dashrightarrow \cM_{I,I'}\to \Sigma_{I,I'},
\end{equation*}
given by $\ssp\circ\pr^{-1}$.

\begin{corollary}\label{cor:chart}
$\phi_{I,I'}:\bC^{d(I,I')}\dashrightarrow \Sigma_{I,I'}$ is birational.
\end{corollary}

\begin{proof}
    Upon restriction to $W_{I,I'}$, the map $\phi_{I,I'}$ is defined everywhere, and injective by Proposition \ref{prop:solve_for_*} and Corollary \ref{cor:matrix_injective}. By Lemma \ref{lem:dII'} below, $\phi_{I,I'}$ is a rational map between irreducible varieties of the same dimension. The conclusion follows.
\end{proof}

\begin{lemma}\label{lem:dII'}
We have $d(I,I')=\dim(\Sigma_{I,I'})$.
\end{lemma}

\begin{proof}
If $(I,I')$ is not saturated, then the number of $+$-entries of $M_{I,I'}$ along the lattice path $P_{\NE}$ is exactly $n-k$, one for each of the columns $\overline{n-k+1},\ldots,\bar2$. The number of $+$-entries in the inner box $M_{I,I'}^{\inbox}$ is
\begin{equation*}
    1+2+\cdots+k-j.
\end{equation*}
The entry in the upper right corner of $M_{I,I'}^{\inbox}$ is a 1-entry, so we have not counted it in either case. There are no $+$-entries along $P_{\SW}$. In total, we have 
\begin{align*}
    d(I,I')&=[1+2+\cdots+(k-1)]+n-j\\
    &=[1+\cdots+n]-[k+\cdots+(n-1)]-j\\
    &=\dim(\OG(n,2n+1))-w(I)-w(I')\\
    &=\dim(\Sigma_{I,I'}).
\end{align*}

If $(I,I')$ is saturated, then the number of $+$-entries of $M_{I,I'}$ is $n$, one for each of the columns $0,\barn,\ldots,\bar2$, which again equals the dimension of $\Sigma_{I,I'}$.
\end{proof}

\section{Degenerations of torus orbits}\label{sec:degen}

\subsection{Overview and example}\label{sec:degen_overview}

We now begin to discuss our degeneration of a general torus orbit closure $Z\subset\OG(n,2n+1)$. It is constructed iteratively, indexed by the tree $\cT_n$ (Definition \ref{def:T_n}).

Let $(I,I')$ be a non-saturated allowed pair, and write
\begin{align*}
    \ell(I,I')&=(I,I'_+),\\
    r(I,I')&=(I_+,I').
\end{align*}
as in \S\ref{sec:allowed_pairs}. We write $j=j(I,I')$ and $k=k(I,I')$ throughout.

Let $\Lambda\in\Sigma_{I,I'}$ be a general point. We will specify only later the precise sense in which $\Lambda$ is required to be general in Definition \ref{def:nicest}. Adopting the notation of \S\ref{sec:prelim_degen}, we will construct a degeneration
\begin{equation*}
    \pi_{I,I'}:\cZ\subset\OG(n,2n+1)\times U\to U
\end{equation*}
whose general fiber is $Z_{\Lambda}$, and whose special fiber consists of the torus orbits $Z_{\Lambda_\ell}$ and $Z_{\Lambda_r}$, where $\Lambda_\ell\in \Sigma_{\ell(I,I')}$ and $\Lambda_r\in \Sigma_{r(I,I')}$ are general points. Then, we will iterate the degeneration downward through $\cT_n$, until reaching, for all saturated allowed pairs $(I,I')$, torus orbits $Z_{I,I'}$ which are in fact equal to the Richardson variety $\Sigma_{I,I'}$.

We first describe the degeneration $\pi_{I,I'}$ by way of example. 
As in Example \ref{eg:innerbox}, take $n = 6$, $I = \emptyset, I' = \{3\}$. We will see that the case where $I$ is empty and $|I'|\le 1$ is in some sense the main one, because our degeneration will only have an interesting effect in the inner box of a matrix $A\in\cM_{I,I'}$.

Consider a 1-parameter family of matrices
\begin{equation*}
    A(t)=
    \begin{bmatrix}
0 & 0 & 0 & 0 & 0 & 0 & 0 & 0 & 0 & a^1_{\bar4}\textcolor{red}{t} & a^1_{\bar3} & a^1_{\bar2} & 1 \\
0 & 0 & 0 & 0 & 0 & a^2_6 & a^2_0 & a^2_{\bar6} & a^2_{\bar5} & a^2_{\bar4} & a^2_{\bar3} & 1 & 0\\
0 & 0 & 0 & 0 & a^3_5 & a^3_6 & a^3_0 & a^3_{\bar6} & a^3_{\bar5} & a^3_{\bar4} & 1  & 0 & 0 \\
0 & 0 & a^4_3 & a^4_4 & a^4_5 & a^4_6 & a^4_0 & a^4_{\bar6} & a^4_{\bar5} & 1 & 0 & 0 & 0 \\
0 & a^5_2 & a^5_3 & a^5_4 & a^5_5 & a^5_6 & a^5_0 & a^5_{\bar6} & 1 & 0 & 0 & 0 & 0 \\
a^6_1 & a^6_2 & a^6_3 & a^6_4 & a^6_5 & a^6_6 & a^6_0 & 1 & 0 & 0 & 0 & 0 & 0 \\
    
    \end{bmatrix}
    \in\cM_{\emptyset,\{3\}}
\end{equation*}
where $t\in\bC-\{0\}$. Write also $\Lambda(t)$ for the row-span of $A(t)$. Then, the subscheme (degeneration) $\cZ\subset \OG(6,13)\times U$ will be defined to be the closure of $T\cdot \Lambda(t)$, where $t$ ranges over all (or the generic) $t\neq0$ in a suitable Zariski open set $U\subset\bC$. In particular, for any fixed $c\in\bC-\{0\}$, the restriction of $\cZ$ to the fiber over $t=c$ \emph{contains} the orbit closure $Z_{\Lambda(c)}$. We wish to consider the flat limit of $Z_{\Lambda(t)}$ as $t\to 0$, that is, the special fiber $\cZ_0$.

At $t=1$, if the $+$-entries
\begin{equation*}
 a^1_{\bar4},a^1_{\bar3},a^1_{\bar2},a^2_0,a^2_{\bar6},a^2_{\bar5},a^2_{\bar4},a^2_{\bar3},a^3_0,a^3_{\bar6},a^3_{\bar5}, a^3_{\bar4},a^4_0,a^4_{\bar6},a^4_{\bar5},a^5_0,a^5_{\bar6},a^6_0
\end{equation*}
are sufficiently general, then the $*$-entries 
\begin{equation*}
    a^2_6, a^3_5 , a^3_6,a^4_3 , a^4_4 , a^4_5 , a^4_6,a^5_2 , a^5_3 , a^5_4 , a^5_5,a^5_6,a^6_1,a^6_2 , a^6_3 , a^6_4 , a^6_5,a^6_6,
\end{equation*}
may be solved for uniquely. More precisely, the $*$-entries comprise the unique solution to the system of equations given by the orthogonality of each pair of rows of $A(1)$, by Proposition \ref{prop:solve_for_*_box}. 

More generally, viewing the $a^p_{\barq}$ as fixed and $t$ as varying, the $*$-entries are given by rational functions in $t$, obtained by solving the system of equations given by the orthogonality of each pair of rows of $A(t)$. Any pole of a rational function $a^p_{\barq}$ must be a root of the determinant of one of the submatrices $A_2,\ldots,A_6$, viewed as functions of $t$, from the proof of Proposition \ref{prop:solve_for_*_box}. (See also Example \ref{eg:innerbox} for a depiction of these submatrices in the case at hand.) Write also $A_2(0),\ldots,A_6(0)$ for the specializations of these submatrices to $t=0$.

At $t=0$, we claim that there is also a unique solution to this system of equations. That is, assuming the complex numbers $a^p_{\barq}$ were general, there is a unique matrix
\begin{equation*}
    A(0)=
    \begin{bmatrix}
0 & 0 & 0 & 0 & 0 & 0 & 0 & 0 & 0 & \textcolor{red}{0} & a^1_{\bar3} & a^1_{\bar2} & 1 \\
0 & 0 & 0 & 0 & 0 & \alpha^2_6 & a^2_0 & a^2_{\bar6} & a^2_{\bar5} & a^2_{\bar4} & a^2_{\bar3} & 1 & 0\\
0 & 0 & 0 & 0 & \alpha^3_5 & \alpha^3_6 & a^3_0 & a^3_{\bar6} & a^3_{\bar5} & a^3_{\bar4} & 1  & 0 & 0 \\
0 & 0 & \alpha^4_3 & \alpha^4_4 & \alpha^4_5 & \alpha^4_6 & a^4_0 & a^4_{\bar6} & a^4_{\bar5} & 1 & 0 & 0 & 0 \\
0 & \alpha^5_2 & \alpha^5_3 & \alpha^5_4 & \alpha^5_5 & \alpha^5_6 & a^5_0 & a^5_{\bar6} & 1 & 0 & 0 & 0 & 0 \\
\alpha^6_1 & \alpha^6_2 & \alpha^6_3 & \alpha^6_4 & \alpha^6_5 & \alpha^6_6 & a^6_0 & 1 & 0 & 0 & 0 & 0 & 0 \\
    
    \end{bmatrix}
    ,
\end{equation*}
whose rows are orthogonal. To see this, note first that by the orthogonality of rows 1 and 4 of $A(0)$, it must be the case that $\alpha^4_3=0$. We are now in the setting of Proposition \ref{prop:solve_for_*_box}, for $(I,I')=\ell(I,I')=(\emptyset,\{4\})$. Recalling the proof of Proposition \ref{prop:solve_for_*_box} in this case, we may solve uniquely for the remaining $*$-entries $\alpha^p_q$ as long as the determinants of the submatrices $A(0)_2,\ldots,A(0)_6$, do not vanish. (Here, we contrast the notation $A(0)_p$, denoting a submatrix coming from the proof of Proposition \ref{prop:solve_for_*_box} in the case $(I,I')=(\emptyset,\{4\})$, with the earlier notation $A_p(0)$, which denotes a sub-matrix coming from the proof of Proposition \ref{prop:solve_for_*_box} in the case $(I,I')=(\emptyset,\{3\})$.) The proof of Proposition \ref{prop:solve_for_*_box} shows that this non-vanishing indeed holds for a general choice of the $+$-entries $a^p_{\barq}$.

For $p=2,3,5,6$, we have $A(0)_p=A_p(0)$. For $p=4$, we have
\begin{equation*}
A_4(0)=\left[
\begin{array}{c|c} 
0 & a^1_{\bar3} \\ 
\hline A(0)_4 &0 \\

\end{array}
\right]
\end{equation*}
where the single $+$-entry $a^1_{\bar3}$ may be assumed to be non-zero. It follows that all five of the matrices $A(0)_2, A(0)_3, A(0)_4, A(0)_5, A(0)_6$ have non-vanishing determinant, and therefore that the functions $a^p_q$ are defined at $t=0$.

We conclude that
\begin{equation*}
    A_\ell:=\lim_{t\to 0}A(t)=
    \begin{bmatrix}
0 & 0 & 0 & 0 & 0 & 0 & 0 & 0 & 0 & \textcolor{red}{0} & a^1_{\bar3} & a^1_{\bar2} & 1 \\
0 & 0 & 0 & 0 & 0 & \alpha^2_6 & a^2_0 & a^2_{\bar6} & a^2_{\bar5} & a^2_{\bar4} & a^2_{\bar3} & 1 & 0\\
0 & 0 & 0 & 0 & \alpha^3_5 & \alpha^3_6 & a^3_0 & a^3_{\bar6} & a^3_{\bar5} & a^3_{\bar4} & 1  & 0 & 0 \\
0 & 0 & \textcolor{red}{0} & \alpha^4_4 & \alpha^4_5 & \alpha^4_6 & a^4_0 & a^4_{\bar6} & a^4_{\bar5} & 1 & 0 & 0 & 0 \\
0 & \alpha^5_2 & \alpha^5_3 & \alpha^5_4 & \alpha^5_5 & \alpha^5_6 & a^5_0 & a^5_{\bar6} & 1 & 0 & 0 & 0 & 0 \\
\alpha^6_1 & \alpha^6_2 & \alpha^6_3 & \alpha^6_4 & \alpha^6_5 & \alpha^6_6 & a^6_0 & 1 & 0 & 0 & 0 & 0 & 0 \\
    
    \end{bmatrix}
    \in\cM_{\emptyset,\{4\}},
\end{equation*}
where $\alpha^p_q=\lim_{t\to 0}a^p_q$. The rows of $A_\ell$ span an isotropic subspace $\Lambda_\ell\in\cM_{\emptyset,\{4\}}$, and
\begin{equation*}
    \overline{T\cdot \Lambda_\ell}=Z_{\Lambda_\ell}\subset \cZ_0,
\end{equation*}
because $\cZ$ is $T$-invariant and $\Lambda_\ell=\lim_{t\to 0}\Lambda(t)$ is a limit of subspaces in the general fiber.

In order to find a second component of the special fiber $\cZ_0$, we translate $\Lambda(t)$ by an element of $T$. Let
\begin{equation*}
    \vectt=(t^{-1},t^{-1},t^{-1},1,1,1)\in(\bC^{*})^6,
\end{equation*}
and consider
 \begin{equation*}
   \vectt\cdot A(t)=
    \begin{bmatrix}
0 & 0 & 0 & 0 & 0 & 0 & 0 & 0 & 0 & a^1_{\bar4}\textcolor{red}{t} & \textcolor{blue}{t}a^1_{\bar3} & \textcolor{blue}{t}a^1_{\bar2} & \textcolor{blue}{t} \\
0 & 0 & 0 & 0 & 0 & a^2_6 & a^2_0 & a^2_{\bar6} & a^2_{\bar5} & a^2_{\bar4} & \textcolor{blue}{t}a^2_{\bar3} & \textcolor{blue}{t} & 0\\
0 & 0 & 0 & 0 & a^3_5 & a^3_6 & a^3_0 & a^3_{\bar6} & a^3_{\bar5} & a^3_{\bar4} & \textcolor{blue}{t}  & 0 & 0 \\
0 & 0 & \textcolor{blue}{t^{-1}}a^4_3 & a^4_4 & a^4_5 & a^4_6 & a^4_0 & a^4_{\bar6} & a^4_{\bar5} & 1 & 0 & 0 & 0 \\
0 & \textcolor{blue}{t^{-1}}a^5_2 & \textcolor{blue}{t^{-1}}a^5_3 & a^5_4 & a^5_5 & a^5_6 & a^5_0 & a^5_{\bar6} & 1 & 0 & 0 & 0 & 0 \\
\textcolor{blue}{t^{-1}}a^6_1 & \textcolor{blue}{t^{-1}}a^6_2 & \textcolor{blue}{t^{-1}}a^6_3 & a^6_4 & a^6_5 & a^6_6 & a^6_0 & 1 & 0 & 0 & 0 & 0 & 0 \\
    
    \end{bmatrix}
    ,
\end{equation*}
 for $t\neq 0$. The scaling factors coming from the action of $\vectt$ are marked in blue, to contrast with the red factor of $t$ in the $(1,\bar4)$ position, which comes from the degeneration $A(t)$. These factors of $t$ play different roles, but are the same variable. By the orthogonality of rows 1 and 4 of $A(t)$, we have
 \begin{equation*}
     a^4_3=-\frac{a^4_4a^1_{\bar4}}{a^1_{\bar3}}t.
 \end{equation*}
 Thus, the entry $\textcolor{blue}{t^{-1}}a^4_3$ may be replaced by $(a_3^4)':=-\frac{a^4_4a^1_{\bar4}}{a^1_{\bar3}}$.

 Scaling the \emph{rows} of a matrix by a non-zero scalar does not change its row span, so for $t\neq0$, the row span $\vectt\cdot \Lambda(t)$ of $\vectt\cdot A(t)$ equals that of
 \begin{equation*}
 A^\dagger_r(t):=
    \begin{bmatrix}
0 & 0 & 0 & 0 & 0 & 0 & 0 & 0 & 0 & a^1_{\bar4} & a^1_{\bar3} & a^1_{\bar2} & 1 \\
0 & 0 & 0 & 0 & 0 & a^2_6 & a^2_0 & a^2_{\bar6} & a^2_{\bar5} & a^2_{\bar4} & \textcolor{blue}{t}a^2_{\bar3} & \textcolor{blue}{t} & 0\\
0 & 0 & 0 & 0 & a^3_5 & a^3_6 & a^3_0 & a^3_{\bar6} & a^3_{\bar5} & a^3_{\bar4} & \textcolor{blue}{t}  & 0 & 0 \\
0 & 0 & (a_3^4)' & a^4_4 & a^4_5 & a^4_6 & a^4_0 & a^4_{\bar6} & a^4_{\bar5} & 1 & 0 & 0 & 0 \\
0 & a^5_2 & a^5_3 & \textcolor{blue}{t}a^5_4 & \textcolor{blue}{t}a^5_5 & \textcolor{blue}{t}a^5_6 & \textcolor{blue}{t}a^5_0 & \textcolor{blue}{t}a^5_{\bar6} & \textcolor{blue}{t} & 0 & 0 & 0 & 0 \\
a^6_1 & a^6_2 & a^6_3 & \textcolor{blue}{t}a^6_4 & \textcolor{blue}{t}a^6_5 & \textcolor{blue}{t}a^6_6 & \textcolor{blue}{t}a^6_0 & \textcolor{blue}{t} & 0 & 0 & 0 & 0 & 0 \\
    
    \end{bmatrix}
    .
\end{equation*}
Note that the first row has been restored to the first row of $A(1)$. Sending $t\to 0$ in this new matrix, we have that $\lim_{t\to 0}\vectt\cdot \Lambda(t)$ is equal to the row span of the matrix
 \begin{equation*}
  A^\dagger_r:=
    \begin{bmatrix}
0 & 0 & 0 & 0 & 0 & 0 & 0 & 0 & 0 & a^1_{\bar4} & a^1_{\bar3} & a^1_{\bar2} & 1 \\
0 & 0 & 0 & 0 & 0 & \alpha^2_6 & a^2_0 & a^2_{\bar6} & a^2_{\bar5} & a^2_{\bar4} & \textcolor{blue}{0} & \textcolor{blue}{0} & 0\\
0 & 0 & 0 & 0 & \alpha^3_5 & \alpha^3_6 & a^3_0 & a^3_{\bar6} & a^3_{\bar5} & a^3_{\bar4} & \textcolor{blue}{0}  & 0 & 0 \\
0 & 0 & (\alpha^4_3)' & \alpha^4_4 & \alpha^4_5 & \alpha^4_6 & a^4_0 & a^4_{\bar6} & a^4_{\bar5} & 1 & 0 & 0 & 0 \\
0 & \alpha^5_2 & \alpha^5_3 & \textcolor{blue}{0} & \textcolor{blue}{0} & \textcolor{blue}{0} & \textcolor{blue}{0} & \textcolor{blue}{0} & \textcolor{blue}{0} & 0 & 0 & 0 & 0 \\
\alpha^6_1 & \alpha^6_2 & \alpha^6_3 & \textcolor{blue}{0} & \textcolor{blue}{0} & \textcolor{blue}{0} & \textcolor{blue}{0} & \textcolor{blue}{0} & 0 & 0 & 0 & 0 & 0 \\
    
    \end{bmatrix}
    ,
\end{equation*}
as long as the rows of $A^\dagger_r$ are linearly independent. We have replaced the $*$-entries $a^p_q$ with their limits $\alpha^p_q$ as $t\to 0$; these are the same $\alpha^p_q$ that appeared earlier in $A_\ell$. The new entry $(\alpha_3^4)'$ is by definition equal to
\begin{equation*}
    \lim_{t\to 0}(a^4_3)'=-\frac{\alpha^4_4a^1_{\bar4}}{a^1_{\bar3}}.
\end{equation*}
In particular, by Proposition \ref{prop:*_nonzero} applied to $(I,I')=(\emptyset,\{4\})$, the left-most entry of every row is non-zero, so the rows of $A^\dagger_r$ are indeed linearly independent.

If the $a^{p}_{\barq}$ are general, then suitable row operations may be applied so that the right-most entries of the first four rows of $A^\dagger_r$ become equal to 1, and furthermore, that these 1's appear in distinct columns, moving from top to bottom, right to left. Recall that, by Proposition \ref{prop:*_nonzero}, we have $\alpha^5_2\neq0$. It follows from the orthogonality of rows 1 and 5 that $\alpha^5_3\neq0$, so the 5th row can be scaled so that the rightmost non-zero entry, appearing in column 3, equals 1. Repeating the argument, we have $\alpha^6_1\neq0$, and after suitable row operations, we may arrange for the rightmost non-zero entry in the 6th row to occur in column 2 and to be equal to 1. Therefore, $A^\dagger_r$ is row-equivalent to
 \begin{equation*}
 A_r:=
    \begin{bmatrix}
0 & 0 & 0 & 0 & 0 & 0 & 0 & 0 & 0 & a^1_{\bar4} & a^1_{\bar3} & a^1_{\bar2} & 1 \\
0 & 0 & 0 & 0 & 0 & (\alpha^2_6)' & (a^2_0)' & (a^2_{\bar6})' & (a^2_{\bar5})' & 1 & \textcolor{blue}{0} & \textcolor{blue}{0} & 0\\
0 & 0 & 0 & 0 & (\alpha^3_5)' & (\alpha^3_6)' & (a^3_0)' & (a^3_{\bar6})' & 1 & \textcolor{blue}{0} & \textcolor{blue}{0}  & 0 & 0 \\
0 & 0 & (\alpha^4_3)'' & (\alpha^4_4)' & (\alpha^4_5)' & (\alpha^4_6)' & (a^4_0)' & 1 & \textcolor{blue}{0} & \textcolor{blue}{0} & 0 & 0 & 0 \\
0 & (\alpha^5_2)' & 1 & \textcolor{blue}{0} & \textcolor{blue}{0} & \textcolor{blue}{0} & \textcolor{blue}{0} & \textcolor{blue}{0} & \textcolor{blue}{0} & 0 & 0 & 0 & 0 \\
(\alpha^6_1)' & 1 & \textcolor{blue}{0} & \textcolor{blue}{0} & \textcolor{blue}{0} & \textcolor{blue}{0} & \textcolor{blue}{0} & \textcolor{blue}{0} & 0 & 0 & 0 & 0 & 0 \\
    
    \end{bmatrix}
    .
\end{equation*}
The entries of $A_r$ are rational functions of the entries of $A^{\dagger}_r$, which in turn are rational functions of the original $a^p_{\barq}$. 

Finally, if the original $a^p_{\barq}$ are suitably general, then so is the tuple of new $+$-entries
\begin{equation*}
(a^1_{\bar4},a^1_{\bar3} , a^1_{\bar2},(a^2_0)',  (a^2_{\bar6})' , (a^2_{\bar5})',(a^3_0)' , (a^3_{\bar6})',(a^4_0)')\in \bC^{d(\{4,5\},\{3\})}=\bC^{d(I_+,I')},
\end{equation*}
see Lemma \ref{lem:pr_r_is_dominant}. Therefore, $A_r$ is a general element of $\cM_{\{4,5\},\{3\}}$, whose row span $\Lambda_r$ satisfies
\begin{equation*}
    \Lambda_r:=\lim_{t\to 0}\vectt \cdot \Lambda(t)\in\Sigma_{\{4,5\},\{3\}}.
\end{equation*}
In particular, we have $Z_{\Lambda_r}\subset\cZ_0$.

To summarize, the 1-parameter family of matrices $A(t)$, obtained by replacing one coordinate $a^1_{\bar4}$ in a general point of $\bC^{d(I,I')}$ with $a^1_{\bar4}t$, gives rise to a degeneration $\pi_{I,I'}$ with general fiber $Z_{\Lambda(t)}$. The flat limit of the degeneration was found to contain the orbit closures of subspaces $\Lambda_\ell\in\Sigma_{(I,I'_+)}$ (by computing the naive limit of the level of matrices $A(t)$) and $\Lambda_r\in\Sigma_{(I_+,I')}$ (by computing the limit of a translate of $A(t)$).

We have not argued that the orbit closures $Z_{\Lambda_\ell}$ and $Z_{\Lambda_r}$ are \emph{components} of the special fiber of the degeneration, but we will see that follows for dimension reasons after iterating the degeneration. More seriously, we have not argued that no other orbit closures appear as components of the special fiber. This is only proven in \S\ref{sec:moment_polytopes}, where the criterion of Corollary \ref{cor:moment_map_union} is verified, along with the fact that the components of the special fiber appear with multiplicity 1.

\subsection{The degeneration in general}\label{sec:degen_real}

We first make precise the needed generality hypotheses on $\Lambda\in\Sigma_{I,I'}$. Our degeneration will be defined with respect to \emph{nicest} $\Lambda$, Definition \ref{def:nicest}.

\begin{definition}\label{def:nice}
We say that $\Lambda\in\Sigma_{I,I'}$ is \emph{nice} if the birational map
\begin{equation*}
    \phi_{I,I'}:\bC^{d(I,I')}\dashrightarrow\Sigma_{I,I'},
\end{equation*}
constructed in the previous section, is an isomorphism over a neighborhood of $\Lambda$.
\end{definition}
The notion of niceness depends on the allowed pair $(I,I')$. If $\Lambda$ is nice, then it may be viewed equally well as a matrix $A\in\cM_{I,I'}$ or a point $a
\in \bC^{d(I,I')}$ with non-zero coordinates, corresponding to the $+$-entries of $A$. We also say that $a\in \bC^{d(I,I')}$ or $A\in\cM_{I,I'}$ is nice if it corresponds to a nice $\Lambda\in \Sigma_{I,I'}$. We will freely identify nice $a,A,\Lambda$, but reserve each of these letters to indicate the type of object in question: a vector in $\bC^{d(I,I')}$, a matrix, or a subspace.

\begin{definition}
    We say that $\Lambda\in\Sigma_{I,I'}$ is \emph{very nice} if it is nice and, for every maximal admissible subset $S\in \mAdm_n$, either $S\in\cD(\Lambda)$ or, it is the case that $S\notin\cD(\Lambda_\bullet)$ for \emph{every} $\Lambda_\bullet\in \Sigma_{I,I'}$.

    Similarly, we say that $a\in\bC^{d(I,I')}$ or $A\in\cM_{I,I'}$ are very nice if they correspond to very nice $\Lambda$.
\end{definition}

Recall that $\cD(\Lambda)$ is the set of maximal admissible $S$ for which the $n\times n$ submatrix $A_S$ has full rank. For a fixed $S\in\mAdm_n$, the locus of (nice) $\Lambda\in\Sigma_{I,I'}$ for which $S\in\cD(\Lambda)$ is open, but possibly empty. The very nice locus is the intersection of the non-empty open loci, and in particular non-empty. Equivalently, a nice $\Lambda$ is very nice if it has the largest possible (ordered by inclusion) polytope $\cP(\Lambda)$ among all points of $\Sigma_{I,I'}$.

\begin{lemma}\label{lem:very_nice_saturated}
    Suppose that $(I,I')$ is saturated and that $\Lambda\in \Sigma_{I,I'}$ is nice. Then:
    \begin{enumerate}
    \item[(a)] $\Lambda$ is very nice. 
    \item[(b)] $Z_{\Lambda}=\Sigma_{I,I'}$. In particular, $\dim(Z_\Lambda)=n$.
    \end{enumerate}
\end{lemma}

\begin{proof}

    We claim that $T$ acts transitively on the nice locus of $\Sigma_{I,I'}$. By Proposition \ref{prop:solve_for_*_saturated}, a nice subspace in $\Lambda\in \Sigma_{I,I'}$ is given by the row span of a matrix $A$ obtained by replacing the $+$-entries of $M_{I,I'}$ with \emph{any} non-zero complex numbers, and solving uniquely for the $*$-entries. Thus, we may identify a nice point of $\Sigma_{I,I'}$ with the tuple
    \begin{equation*}
        a=(a_0,a_{\barn},\ldots,a_{\bar2})\in (\bC^{*})^n=(\bC^{*})^{d(I,I')}
    \end{equation*}
    of $+$-entries appearing in $A$.

    For $q\in[2,n]$, the element $(t_1,\ldots,t_n)\in T$ acts on the coordinate $a_{\barq}$ by $\frac{t_{r(q)}}{t_q}$, where $\overline{r(q)}\in\overline{[q-1]}$ is the column containing the unique 1-entry of $A$ appearing in the same row as $a_{\barq}$. The same is true if $\barq=0$, in which case we interpret the denominator $t_q$ to be equal to 1. The corresponding weights $f_{r(q)}-f_q\in \mathfrak{t}^{*}$ form an integral basis of the weight lattice of $T$. Therefore, after a change of variables, the action of $T$ on the nice locus of $\Sigma_{I,I'}$, identified with $(\bC^{*})^n$, is the standard one, and in particular transitive.

    Now, because the $T$-action preserves the polytope $\cP(\Lambda)$, every nice $\Lambda\in\Lambda_{I,I'}$ has the same polytope. Therefore, we obtain (a). Furthermore, the nice locus of $\Sigma_{I,I'}$ is contained in $Z_\Lambda$, and because the nice locus is dense, we have $\Sigma_{I,I'}\subset Z_\Lambda$. The opposite inclusion is obvious, so we have (b).
\end{proof}

In our degeneration, we will modify the entries of a matrix $A\in\cM_{I,I'}$ until it becomes either a matrix $A_\ell\in \cM_{I,I'_+}$ or a matrix $A_r\in\cM_{I_+,I'}$. The $+$-entries of $M_{I,I'_+}$ and $M_{I_+,I'}$ are proper subsets of the entries of the $+$-entries of $M_{I,I'}$. In particular, we have a natural projection map
\begin{equation*}
    \pr_\ell:\bC^{d(I,I')}\to \bC^{d(I,I'_+)}
\end{equation*}
given by forgetting the coordinate corresponding to the leftmost $+$-entry in the top row of the inner box of a matrix $A\in\cM_{I,I'}$. We refer to this coordinate as the \emph{active} coordinate of $a$, which is identified with an entry of $A$. Using the re-indexing of the rows and columns of $A^{\inbox}$ from \S\ref{sec:charts}, the active coordinate is $b^1_{\overline{k-j+1}}$ if $j>0$, and $b^1_0$ otherwise. (Note that column $\overline{k-j+1}$ of $A^{\inbox}$ is column $\overline{n-j+1}$ of $A$).

We also have a projection map
\begin{equation*}
    \pr_r:\bC^{d(I,I')}\dashrightarrow \bC^{d(I_+,I')},
\end{equation*}
but $\pr_r$ will only be a rational map. Recall from \S\ref{sec:degen_overview} that passing from $\cM_{I,I'}$ to $\cM_{I_+,I'}$ requires row reduction to re-normalize the matrix, after replacing some entries of $A\in\cM_{I,I'}$ with 0.

The map $\pr_r$ is defined as follows. First, view the coordinates of $a\in \bC^{d(I,I')}$ as placed in the columns $0,\barn,\ldots,\bar1$ of a matrix $A\in \cM_{I,I'}$. (The columns $1,\ldots,n$ and the orthogonality do not play an essential role in the definition.)
\begin{enumerate}
    \item For the $+$-entries $a^p_q$ above $A^{\inbox}$ or in the first row of $A^{\inbox}$, which persist in $\cM_{I_+,I'}$, define $\pr_r$ to be the identity on the corresponding coordinates.
    \item For the $+$-entries $a^p_q$ in $A^{\inbox}$ in the $(j+2)$-th row or below, which become 0-entries in $\cM_{I_+,I'}$, define $\pr_r$ to forget the corresponding coordinates.
\end{enumerate}
Next, in rows $2,3,\ldots,j+1$ of $A^{\inbox}$ (the only remaining rows of $A$ in which $+$-entries appear), replace all entries in columns $\overline{k-j},\overline{k-j-1},\ldots,\bar2$ with 0. Here, we again use the re-indexing for rows and columns of $A^{\inbox}$ so that the row set is $1,\ldots,k$ and the column set is $0,\bark,\ldots,\bar1$. Then, apply the unique row operations to these modified rows so that the right-most non-zero entries are re-normalized to 1, and appear in columns $\overline{k-j+1},\ldots,\bark$, from top to bottom. Finally, extract the $+$-entries, where the resulting matrix is viewed in $\cM_{I_+,I'}$.

\begin{example}\label{eg:pr_r}
As in the previous section, take $n=6$ and $(I,I')=(\emptyset,\{3\})$, so that $A^{\inbox}=A$. 

The columns $0,\bar6,\ldots,\bar1$ of $A$ are shown. First, the prescribed entries in rows $2,3,4$ and columns $\bar3,\bar2$ are replaced with 0. The entries in rows 5 and 6, which become 0-entries in $\cM_{I_+,I'}$, are also depicted to become 0. Then, row operations are applied to rows $2,3,4$ so that 1's appear in the required positions; if the original $+$-entries are general, then this can be performed in a unique way. 
\begin{align*}
    \begin{bmatrix}
 0 & 0 & 0 & a^1_{\bar4} & a^1_{\bar3} & a^1_{\bar2} & 1 \\
 a^2_0 & a^2_{\bar6} & a^2_{\bar5} & a^2_{\bar4} & a^2_{\bar3} & 1 & 0\\
 a^3_0 & a^3_{\bar6} & a^3_{\bar5} & a^3_{\bar4} & 1  & 0 & 0 \\
 a^4_0 & a^4_{\bar6} & a^4_{\bar5} & 1 & 0 & 0 & 0 \\
 a^5_0 & a^5_{\bar6} & 1 & 0 & 0 & 0 & 0 \\
 a^6_0 & 1 & 0 & 0 & 0 & 0 & 0 
    \end{bmatrix}
    &\rightarrow
    \begin{bmatrix}
 0 & 0 & 0 & a^1_{\bar4} & a^1_{\bar3} & a^1_{\bar2} & 1 \\
 a^2_0 & a^2_{\bar6} & a^2_{\bar5} & a^2_{\bar4} & \textcolor{blue}{0} & \textcolor{blue}{0} & 0\\
 a^3_0 & a^3_{\bar6} & a^3_{\bar5} & a^3_{\bar4} & \textcolor{blue}{0}  & 0 & 0 \\
 a^4_0 & a^4_{\bar6} & a^4_{\bar5} & 1 & 0 & 0 & 0 \\
 \textcolor{blue}{0}  & \textcolor{blue}{0}  & \textcolor{blue}{0}  & 0 & 0 & 0 & 0 \\
\textcolor{blue}{0}  & \textcolor{blue}{0}  & 0 & 0 & 0 & 0 & 0 
    \end{bmatrix}
    \\
        &\rightarrow
    \begin{bmatrix}
 0 & 0 & 0 & a^1_{\bar4} & a^1_{\bar3} & a^1_{\bar2} & 1 \\
 (a^2_0)' & (a^2_{\bar6})' & (a^2_{\bar5})' & 1 & \textcolor{blue}{0} & \textcolor{blue}{0} & 0\\
 (a^3_0)' & (a^3_{\bar6})' & 1 & \textcolor{blue}{0} & \textcolor{blue}{0}  & 0 & 0 \\
 (a^4_0)' & 1 & \textcolor{blue}{0} & \textcolor{blue}{0} & 0 & 0 & 0 \\
 \textcolor{blue}{0}  & \textcolor{blue}{0}  & \textcolor{blue}{0}  & 0 & 0 & 0 & 0 \\
 \textcolor{blue}{0}  & \textcolor{blue}{0}  & 0 & 0 & 0 & 0 & 0     
    \end{bmatrix}
\end{align*}
Finally, $\pr_r$ outputs the 9-tuple
\begin{equation*}
(a^1_{\bar4},a^1_{\bar3} , a^1_{\bar2},(a^2_0)',  (a^2_{\bar6})' , (a^2_{\bar5})',(a^3_0)' , (a^3_{\bar6})',(a^4_0)')\in \bC^{d(\{4,5\},\{3\})}.
\end{equation*}

More generally, for any allowed pair $(I,I')$ with $k(I,I')=6$ and $j(I,I')=3$, the above matrix entries appear in the inner box of $A\in\cM_{I,I'}$. The map $\pr_r$ is defined in an identical way, retaining the additional information of the $+$-entries along the lattice path $P_{\NE}$, which are unaffected by replacing $I$ with $I_+$.
\end{example}

\begin{lemma}\label{lem:pr_r_is_dominant}
    The map $\pr_r$ is dominant.
\end{lemma}

\begin{proof}
It is enough to consider the $j\times (j+1)$ submatrix $A'$ of $A^{\inbox}$ in which the transformed entries of the image matrix in $\cM_{I_+,I'}$ appear. We say that a $j\times (j+1)$ matrix $A'$ is ``$\ell$-sparse'' if the first $\ell$ anti-diagonals of $A'$ are all zero, but all entries along the $(\ell+1)$-st anti-diagonal are non-zero. Here, the ``first anti-diagonal'' refers to the bottom right entry of $A'$, the ``second anti-diagonal,'' of length 2, refers to entries immediately above and to the left, and so on. If $\ell=-1$, then any matrix is considered $(-1)$-sparse.

Upon restriction to $A'$, the map $\pr_r$ takes in a matrix which is $\max(-1,2j-k)$-sparse, and applies row operations in the unique way to produce a matrix which is $(j-1)$-sparse. It suffices to show that a general $(j-1)$-sparse matrix $A'$ is row-equivalent to some $\ell$-sparse matrix, for any $\ell<j-1$. However, this is clear upon simply moving the bottom $(j-1)-\ell$ rows of $A'$ to the top.
\end{proof}

\begin{definition}\label{def:nicest}
Let $a\in\bC^{d(I,I')}$ be a very nice point. We say that $a$ is \emph{nicest} if either $(I,I')$ is saturated, or both $\pr_\ell(a)$ and $\pr_r(a)$ are nicest (in particular, $\pr_r(a)$ is defined). We say that $A\in\cM_{I,I'}$ or $\Lambda\in\Sigma_{I,I'}$ is nicest if the corresponding point $a$ is.
\end{definition}

The notion of a nicest subspace makes sense by iterating upward the tree $\cT_n$, starting with the saturated allowed pairs. The nice and very nice loci on $\Sigma_{I,I'}$ are manifestly Zariski open and non-empty. By definition, the nicest locus on $\Sigma_{I,I'}$ is therefore Zariski open and non-empty for saturated $(I,I')$. Because the projection maps $\pr_\ell,\pr_r$ are dominant, if the nicest loci for $\ell((I,I'))$ and $r((I,I'))$ are both Zariski open and non-empty, then the same is true of $(I,I')$. By induction, it follows that the nicest locus is Zariski open and non-empty in $\Sigma_{I,I'}$ for any allowed pair $(I,I')$. 

We are now ready to construct the degeneration $\pi_{I,I'}$. Let $a\in\bC^{d(I,I')}$ be a nicest point. For $t\in\bC$, let $a(t)\in \bC^{d(I,I')}$ be the point where the active coordinate $b^1_{\overline{k-j+1}}$ is replaced by $b^1_{\overline{k-j+1}}t$ (we interpret $b^1_{\overline{k+1}}$ to be $b^1_0$ if $j=0$). For $t=1$, and therefore for a non-empty Zariski open subset $U\subset \bA^1$, containing the point $t=0$, we have that $a(t)$ is nicest for all $t\in U-\{0\}$. Therefore, we may consider the corresponding families of matrices $A(t)\in\cM_{I,I'}$ and subspaces $\Lambda(t)\in\OG(n,2n+1)$ relative to $U-\{0\}$.

\begin{definition}
We define
    \begin{equation*}
        \pi_{I,I'}:\cZ\subset\OG(n,2n+1)\times U\to U
    \end{equation*}
    to be the closure of the fiberwise torus orbit $T\cdot \Lambda(t)\subset \OG(n,2n+1)\times (U-\{0\})$.
\end{definition}

By construction, $\cZ$ is equivalently the closure of the orbit of $\Lambda(t)$ over the \emph{generic} point, so is automatically flat over $U$. For a closed point $t\in U-\{0\}$, the fiber $\cZ_t$ contains the fiberwise torus orbit closure $Z_{\Lambda(t)}=\overline{T\cdot \Lambda(t)}$ as an irreducible component. For a \emph{general} closed point, we have in fact $\cZ_t=Z_{\Lambda(t)}$. However, at an arbitrary closed point, $\cZ_t$ could in principle contain additional components, necessarily torus orbit closures of different points of $\OG(n,2n+1)$. In particular, this may occur at $t=1$. The very niceness assumption on $a(t)$ rules this out.

\begin{lemma}\label{lem:very_nice_implies_irreducible}
    For every $t\in U-\{0\}$, the fiber $\cZ_t$ of $\pi_{I,I'}$ is irreducible, and hence equals $Z_{\Lambda(t)}$.
\end{lemma}

\begin{proof}
    By construction, the claim holds at a \emph{general} fiber $t=t_1$. By the discussion of \S\ref{sec:prelim_degen}, the fiber over an arbitrary point $t=t_0$ is the union of the torus orbit closure $Z_{\Lambda(t_0)}$ and possibly other torus orbit closures. Because $\Lambda(t_0)$ and $\Lambda(t_1)$ are very nice, the polytopes $\cP(\Lambda(t_0))$ and $\cP(\Lambda(t_1))$ are equal. The Lemma now follows from Corollary \ref{cor:moment_map_union}.
\end{proof}

Thus, $\pi_{I,I'}$ may be viewed as a degeneration of the $t=1$ fiber, which is the torus orbit closure of the nicest subspace $\Lambda\in\Sigma_{I,I'}$ corresponding to $a$, into a union of torus orbit closures on the $t=0$ fiber, given by the flat limit of $Z_{\Lambda(t)}$ as $t\to 0$. The key property of $\pi_{I,I'}$, given below, identifies two general points of the special fiber $\cZ_0\subset\OG(n,2n+1)$.

\begin{proposition}\label{prop:identify_components}
    Let $a_\ell=\pr_\ell(a)$ and $a_r=\pr_r(a)$, which by assumption are nicest points of $\bC^{d(I,I'_+)}$ and $\bC^{d(I_+,I')}$, respectively. Let $\Lambda_\ell,\Lambda_r$ be the corresponding nicest points of $\Sigma_{I,I'_+}$ and $\Sigma_{I_+,I'}$, respectively. 
    
    Then, we have $\Lambda_\ell,\Lambda_r\in \cZ_0$, In particular, we have $Z_{\Lambda_\ell},Z_{\Lambda_r}\subset\cZ_0$. Furthermore, $Z_{\Lambda_\ell},Z_{\Lambda_r}$ are distinct irreducible components of $\cZ_0$.
\end{proposition}

To show that $\Lambda_\ell,\Lambda_r\in \cZ_0$, we have more precisely:

\begin{proposition}\label{prop:identify_limit_subspaces}
Write
\begin{equation*}
    \vectt=(t^{-1},\ldots,t^{-1},1,\ldots,1)\in T=(\bC^*)^n,
\end{equation*}
where the last $j$ coordinates are equal to 1. Then, we have
\begin{align*}
    \Lambda_\ell&=\lim_{t\to 0}\Lambda(t),\\
    \Lambda_r&=\lim_{t\to 0}\vectt\cdot \Lambda(t).
\end{align*}
 The symbol $\lim_{t\to 0}$ denotes flat limit in $\OG(n,2n+1)$.
\end{proposition}

\begin{proof}
The general argument is essentially the same as that in the example of the previous section. Let $A(t)\in\cM_{I,I'}$, where $t\in U-\{0\}$ be the family of matrices associated to $a(t)$, and let $A_\ell$ be the matrix associated to $a_\ell$ and $\Lambda_\ell$. Because $A(t)$ is nicest for all $t\in U-\{0\}$, the $*$-entries $a^p_q$ of $A(t)$ are rational functions of $t$ defined for all $t\in U-\{0\}$, given by the unique solution to the the system of equations in an equal number of variables from Proposition \ref{prop:solve_for_*}. 

The same system of equations also has a unique solution when $t=0$. Indeed, the orthogonality of rows 1 and $j+1$ of $A(t)^{\inbox} $ forces the inner box entry $b^{j+1}_{k-j}$ to be a multiple of $t$, and in particular to evaluate to 0 when $t=0$. Then, by assumption, $\pr_\ell(a)$ is nicest, so there is a unique solution to this system of equations once $b^{j+1}_{k-j}$ is required to be zero, which amounts to the non-vanishing of a collection of determinants of sub-matrices of $A(0)$. It follows that the denominators of the functions $a^p_q$ do not vanish at $t=0$, and the collection of evaluations $a^p_q(0)$ is the collection of $*$-entries of $A_\ell$. We conclude that $\lim_{t\to 0}A(t)=A_\ell$, and hence that $\lim_{t\to 0}\Lambda(t)=\Lambda_\ell$.

Next, we compute $\lim_{t\to 0}\vectt\cdot \Lambda(t)$ by the following procedure:
\begin{enumerate}
    \item[(i)] Compute $\vectt\cdot A(t)$, by scaling columns $1,\ldots,n-j$ by $t^{-1}$ and columns $\overline{n-j},\ldots,\bar1$ by $t$.
    \item[(ii)] Scale all rows above the inner box, and the first row of the inner box of $\vectt\cdot A(t)$ by $t^{-1}$, and scale all rows below row $j+1$ of the inner box by $t$. (This includes all rows below the inner box, but does not include row $j+1$ itself.) Let $A^\dagger_r(t)$ be the resulting matrix.
    \item[(iii)] Set $t=0$. Let $A^\dagger_r(0)$ be the resulting matrix.
    \item[(iv)] Row reduce in the unique possible way so that the rightmost non-zero entry of every row is 1, and these rightmost entries move from top to bottom.
\end{enumerate}
Step (ii) restores all rows above and below the inner box to match $A(1)$; these operations only have a non-trivial effect on the inner box of $A$. Step (ii) also restores the first row of $A^{\inbox}$. Recall that the entry $b^{j+1}_{k-j}$ of $A(t)^{\inbox}$ is a multiple of $t$, so scaling column $n-j$ by $t^{-1}$ cancels this multiple of $t$. All additional entries in columns 1 through $n-j$, which have been multiplied by $t^{-1}$ in steps (i), are multiplied by $t$ in step (ii). Because all of the $*$-entries $a^p_q$ are well-defined at $t=0$, step (iii) makes sense.

Working from top to bottom, step (iv) first has an effect in row 2 of $(A^\dagger_r(0))^{\inbox}$. Between rows 2 and $j+1$, the row operations performed are precisely those in the definition of the rational map $\pr_r$, and by assumption, the resulting collection $\pr_r(a)$ of $+$-entries is nicest.

In the remaining rows $p=j+2,\ldots,k$ of $(A^\dagger_r(0))^{\inbox}$, the leftmost $k-p$ entries equal 0, followed by 
\begin{equation*}
    \beta^p_{k+1-p},\ldots,\beta^p_{k-j},
\end{equation*}
where $\beta^p_q:=b^p_q(0)$, and the remaining entries equal 0. By Proposition \ref{prop:*_nonzero} applied to $(I,I'_+)$, we know that $\beta^p_{k+1-p}\neq0$. In row $p=j+2$, where the only two non-trivial entries are $\beta^p_{k-j-1},\beta^p_{k-j}$, orthogonality with row 1 (of $(A^\dagger_r(0))^{\inbox}$) shows that $\beta^p_{k-j}\neq0$ as well. Thus, row $j+2$ can be scaled so that the entry in the $(p,k-j)$ position becomes 1, and the appropriate multiple can be added to rows $j+3,\ldots,k$ so that the remaining entries in column $k-j$ in rows $j+3$ and below are all 0. In particular, row $j+3$ now only has two non-trivial entries; pairing with row 1 of $(A^\dagger_r(0))^{\inbox}$ shows that they both are non-zero.

Continuing in this fashion, we find in fact that the matrix obtained by row-reducing  $(A^\dagger_r(0))^{\inbox}$ lies in $\cM_{I_+,I'}$: in particular, rows $p=j+2,\ldots,k$, previously of $(A^\dagger_r(0))^{\inbox}$, now each have exactly one $*$- and one 1-entry, and become rightward steps of the lattice path $P_{\SW}$. Because $\pr_r(a)$ is nicest, the new matrix must be precisely the unique $A_r\in\cM_{I_+,I'}$ whose row span is $\Lambda_r$. It follows that $\lim_{t\to 0}\vectt\cdot \Lambda(t)=\Lambda_r$.

\end{proof}

\begin{proof}[Proof of Proposition \ref{prop:identify_components}]
    By construction, the subscheme $\cZ\subset \OG(n,2n+1)\times U$ is $T$-invariant, so Proposition \ref{prop:identify_limit_subspaces} immediately implies that  $\Lambda_\ell,\Lambda_r\in \cZ_0$ and that $Z_{\Lambda_\ell},Z_{\Lambda_r}\subset\cZ_0$. To see that $Z_{\Lambda_\ell}\neq Z_{\Lambda_r}$, note simply that $\Lambda_\ell\notin\Sigma_{I_+,I'}$, but that $Z_{\Lambda_r}\subset \Sigma_{I_+,I'}$. 
    
    Finally, to see that $Z_{\Lambda_\ell}$ and $Z_{\Lambda_r}$ are irreducible components of $\cZ_0$, it suffices to show that they both have the same dimension as $Z_\Lambda$. It is immediate that $\dim(Z_\Lambda)\ge \dim(Z_{\Lambda_\ell}),\dim(Z_{\Lambda_r})$. We claim that, in fact, for any nicest $\Lambda$ with respect to any allowed pair $(I,I')$, we have that $\dim(Z_\Lambda)=n$. Indeed, iterating the degeneration $\pi$ downward in $\cT_n$ until reaching a saturated pair $(I,I')$, we have $\dim(Z_\Lambda)\ge \dim(\Sigma_{I,I'})=n$, by Lemma \ref{lem:very_nice_saturated}. On the other hand, we have the opposite inequality, because $\dim(T)=n$. This completes the proof.
\end{proof}

Let $\Lambda\in\OG(n,2n+1)=\Sigma_{\emptyset,\emptyset}$ be a nicest point. Iterating the degenerations $\pi_{I,I'}$ downward the tree $\cT_n$ gives a sequence of degenerations
\begin{equation*}
Z_\Lambda \rightsquigarrow \bigcup_{I\subset[n-1]}\Sigma_{I,I^c},
\end{equation*}
where $I^c$ denotes the complement of $I\subset[n-1]$. Indeed, the set of leaves of $\cT_n$ is identified with the set of saturated pairs $(I,I')$, which in turn is identified with the set of subsets $I\subset[n-1]$. By Lemma \ref{lem:very_nice_saturated}(b), the orbit closure associated to a nice $\Lambda\in\Sigma_{I,I^c}$ is the entire Richardson variety $\Sigma_{I,I^c}$.

We have not yet ruled out the existence of additional irreducible components in the special fibers of any of the intermediate degenerations $\pi_{I,I'}$, nor have we shown that the components $\Sigma_{I,I^c}$ appear with multiplicity 1. This is done in the next section.

\section{Moment polytopes}\label{sec:moment_polytopes}

\subsection{The polyhedral decomposition}\label{sec:moment_polytopes_statement}

In this section and \S\ref{sec:rank_ineqs}, we show:

\begin{proposition}\label{prop:polytope_union}
    Let $\Lambda\in\Sigma_{\emptyset,\emptyset}=\OG(n,2n+1)$ be a nicest subspace. For each $I\subset[n-1]$, let $\Lambda_I\in\Sigma_{I,I^c}$ be a nice subspace. Then, we have
    \begin{equation*}
        \bigcup_{I\subset[n-1]}\cP(\Lambda_{I})=[0,1]^n=\cP(\Lambda).
    \end{equation*}
\end{proposition}

In particular, the criterion of Corollary \ref{cor:moment_map_union} holds. Our main result, Theorem \ref{thm:degen}, will follow from the degeneration constructed in the previous section, along with the multiplicity 1 statement of Proposition \ref{prop:multiplicity_1}.

The decomposition of $[0,1]^n$ in Proposition \ref{prop:polytope_union} has previously appeared in \cite[Proposition 6.1]{csvy}, in the context of \emph{lattice path delta matroids} arising from the type $C$ Grassmannian $\LG(n,2n)$. One can match our definition of the components of the decomposition with theirs, but we instead include a self-contained proof.

Recall from Definition \ref{def:polytope} that for any subspace $\Lambda\in\OG(n,2n+1)$, the convex polytope $\cP(\Lambda)$ is cut out by the inequalities 
\begin{equation*}
    x(S) \le g_{\Lambda}(S) := \rk_{\Lambda}(S) - |S \cap \overline{[n]}|,
\end{equation*}
for all $S \in \Adm_n$. The definition of $x(S)$ is given in Definition \ref{def:x(S)}.

The inclusion $\bigcup_{I\subset[n-1]}\cP(\Lambda_{I})\subset [0,1]^n$ is immediate. The harder direction is to show that any point $x\in \cP(\Lambda)$ lies in $\cP(\Lambda_I)$ for some $I \subset[n-1]$. Fix, one and for all, such a point $x=(x_1,\ldots,x_n)\in[0,1]^n$. We first identify the candidate subset $I$. For $q=1,\ldots,n$, write
\begin{equation*}
    y_q:=x_1+\cdots+x_q,
\end{equation*}
and note that $0\le y_1\le y_2\le\cdots\le y_n\le n$.

\begin{definition}\label{def:I_associated_to_x}
    Let $x\in[0,1]^n$ be the point above. For all integers $\ell\in[1,\lceil y_n\rceil -1]$, let $\lambda_\ell\in[n-1]$ be the unique integer such that 
    \begin{equation}\label{eq:I'_inequalities}
       y_{n-\lambda_{\ell}}\le\ell< y_{n+1-\lambda_{\ell}}.
    \end{equation}
    Define the set
    \begin{equation*}
        \SC(x)=\{n+1-\lambda_1,\ldots,n+1-\lambda_{s}\}\subset\{2,\ldots,n\},
    \end{equation*}
    where $s=\max(0,\lceil y_n\rceil-1)$.
    
    Then, define $I=\{\lambda_1,\ldots,\lambda_{s}\}\subset[n-1]$. Define also $I'=\{\lambda'_1,\ldots,\lambda'_{s'}\}$, where $\lambda'_1>\cdots>\lambda_{s'}$, to be the complement of $I$ in $[n-1]$.
\end{definition}
Note that
\begin{equation*}
    n+1-\lambda_1<\cdots<n+1-\lambda_{s}.
\end{equation*}
Indeed, for $\ell=1,\ldots,s-1$, we have $y_{n-\lambda_\ell}\le \ell$, hence
\begin{equation*}
    y_{n+1-\lambda_\ell}=y_{n-\lambda_\ell}+x_{n+1-\lambda_\ell}\le \ell+1.
\end{equation*}
On the other hand, we have $\ell+1<y_{n+1-\lambda_{\ell+1}}$, so $y_{n+1-\lambda_\ell}<y_{n+1-\lambda_{\ell+1}}$. Because the $y_q$ are non-decreasing, we therefore have $n+1-\lambda_\ell<n+1-\lambda_{\ell+1}$. The set $\SC(x)$ is the set of special columns (\S\ref{sec:M_II'_structure}) of the lattice path $P_{\SW}$ for the matrix $M_{I,I'}=M_{I,I^c}$.

\begin{example}\label{eg:find_I}
    Let
    \begin{equation*}
        x=(x_1,x_2,x_3,x_4,x_5,x_6,x_7,x_8,x_9)=\left(\frac{9}{10},\frac{9}{10},\frac{1}{10},\frac{9}{10},\frac{9}{10},\frac{9}{10},\frac{1}{10},\frac{1}{10},\frac{9}{10}\right).
    \end{equation*}
    Then, we have
    \begin{align*}
        y_1&<1<y_2,\\
        y_3&<2<y_4,\\
        y_4&<3<y_5,\\
        y_5&<4<y_6,\\
        y_8&<5<y_9.
    \end{align*}
    and $s=\lceil y_9\rceil-1=5$. Thus, we have $\SC(x)=\{2,4,5,6,9\}$ and $(I,I')=(\{1,4,5,6,8\}),\{2,3,7\})$. This is the same saturated pair that appears in Examples \ref{eg:tree_path} and \ref{eg:saturated_path}. In particular, the columns $2,4,5,6,9$ are the special ones along the path $P_{\SW}$.
\end{example}

For $q=1,\ldots,n$, write $x'_q=1-x_q$ and $y'_q=x'_1+\cdots+x'_q=q-y_q$. We will also need the following ``dual'' inequalities to \eqref{eq:I'_inequalities}.

\begin{lemma}\label{lem:dual_ineqs}
    For all integers $m=1,\ldots,s'$, we have
    \begin{equation}\label{eq:I_inequalities}
       y'_{n-\lambda'_{m}}< m\le y'_{n+1-\lambda'_{m}}.
    \end{equation}
\end{lemma}

\begin{proof}
Note that
\begin{equation}\label{eq:disjoint_union}
    [1,n+1-\lambda'_m] =\{1\}\cup \{n+1-\lambda'_1,\ldots,n+1-\lambda'_m\}\cup \{n+1-\lambda_1,\ldots,n+1-\lambda_\ell\},
\end{equation}
where 
\begin{equation*}
    \ell=n-\lambda'_m-m.
\end{equation*}
Moreover, we have $n-\lambda'_m\ge (n+1)-\lambda_\ell$. Therefore, we have
\begin{equation*}
    y'_{n-\lambda'_m}=(n-\lambda'_m)-y_{n-\lambda'_m}\le (n-\lambda'_m)-y_{(n+1)-\lambda_\ell}<(n-\lambda'_m)-\ell=m.
\end{equation*}
Similarly, we have $n+1-\lambda'_m < n+1-\lambda_{\ell+1}$ because $|[1,n+1-\lambda'_m] \cap \SC(x)| = \ell$. Then, $n+1-\lambda'_m \le n-\lambda_{\ell+1}$, and we have
\begin{equation*}
    y'_{n+1-\lambda'_m} = (n+1-\lambda'_m)-y_{n+1-\lambda'_m} \ge (n+1-\lambda'_m)-y_{n-\lambda_{\ell+1}} \ge (n+1-\lambda'_m)-(\ell+1) = m. 
\end{equation*}
\end{proof}

\begin{example}\label{eg:dual_ineqs}
    Continuing Example \ref{eg:find_I}, we have $s'=3$ and 
    \begin{align*}
        y'_2&<1<y'_3,\\
        y'_6&<2<y'_7,\\
        y'_7&<3<y'_8.
    \end{align*}
    The indices $3,7,8$ appearing on the right hand sides correspond to the special columns $\bar3,\bar7,\bar8$ of $P_{\NE}$ in \ref{eg:saturated_path}.
\end{example}

\subsection{Rank inequalities}\label{sec:rank_ineqs}

To prove Proposition \ref{prop:polytope_union}, it suffices to prove:

\begin{proposition}\label{prop:x_in_P(I)}
    Let $x\in[0,1]^n$ be any point, and let $I\subset [n-1]$ be as in Definition \ref{def:I_associated_to_x}. Let $\Lambda_I\in\Sigma_{I,I^c}$ be a nice subspace. Then, for any $S\in\Adm_n$, we have
    \begin{equation*}
        x(S)\le g_{\Lambda_I}(S).
    \end{equation*}
    That is, we have $x\in\cP(\Lambda_I)$.
\end{proposition}

\begin{example}\label{eg:polytope_ineq}
    Continuing the example of Example \ref{eg:find_I}, let $S=\{1,2,\bar4,\bar5,\bar6,8,9\}$. Then, for this admissible set $S$, Proposition \ref{prop:x_in_P(I)} reads
    \begin{equation}\label{eq:rank_ineq_eg}
        (x_1+x_2+x_8+x_9)-(x_4+x_5+x_6)\le \rk(A_S)-3,
    \end{equation}
    where $A=\ssp^{-1}(\Lambda_I)\in\cM_{I,I^c}$, see Example \ref{eg:saturated_path} for the shape of $A$.

    We have by inspection that
    \begin{equation*}
        \rk(A_S)=\rk(A_{\{1,2\}})+\rk(A_{\{8,9\}})+\rk(A_{\{\bar4,\bar5,\bar6\}})=2+2+1.
    \end{equation*}
    It is immediate that $x_1+x_2,x_8+x_9\le 2$. We also have
    \begin{equation*}
        3-(x_4+x_5+x_6)=y'_6-y'_3\le 2-1=1,
    \end{equation*}
    by Example \ref{eg:dual_ineqs}. Adding these three inequalities yields \eqref{eq:rank_ineq_eg}.
\end{example}

In general, we will bound $g_{\Lambda_I}(S)=\rk_{\Lambda_I}(S) - |S \cap \overline{[n]}|$ using the same strategy as in Example \ref{eg:polytope_ineq}.

\begin{lemma}\label{lem:split_rank}
    Let $I \subset [n-1]$ and $\Lambda_I \in \Sigma_{I,I^c}$ be as in Proposition \ref{prop:x_in_P(I)}. Then,
    \begin{equation*}
        \rk_{\Lambda_I}(S) = \rk_{\Lambda_I}(S \cap [n])+ \rk_{\Lambda_I}(S \cap \overline{[n]})
    \end{equation*}
    for all $S \in \Adm_n$.
\end{lemma}

\begin{proof}
Let $A=\ssp^{-1}(\Lambda_I)\in\cM_{I,I^c}$ (see \S\ref{sec:matrices}) be the unique matrix in $\cM_{I,I^c}$ whose row span is equal to $\Lambda_I$. Then, we regard $\rk_{\Lambda_I}(S)$ as the dimension of the column span $\langle A_S\rangle$ of the sub-matrix $A_S$ comprised of columns indexed by $S$.

Let $h_i\in\bC^n$, for $i=1,\ldots,n$, be the column vector whose $i$-th coordinate is 1, and all of whose other coordinates are 0. The 0-th column of $A$ is a non-zero multiple of $h_p$, where $p=|I^c|+1$. Furthermore, we have
\begin{align*}
    \langle A_{S\cap[n]} \rangle &\subset \langle h_p,\ldots,h_n\rangle,\\
    \langle A_{S\cap\overline{[n]}} \rangle &\subset \langle h_1,\ldots,h_p\rangle.
\end{align*}
Therefore, we have
    \begin{equation*}
        \rk_{\Lambda_I}(S) = \rk_{\Lambda_I}(S \cap [n])+ \rk_{\Lambda_I}(S \cap \overline{[n]})
    \end{equation*}
as long as $h_p$ does not lie in both $\langle A_{S\cap[n]} \rangle$ and $\langle A_{S\cap\overline{[n]}} \rangle$.

Suppose for sake of contradiction that $h_p$ lies in the column spans of $A_{S\cap[n]}$ and $A_{S\cap\overline{[n]}}$. We claim first that $S\cap([2,n]\cup\overline{[2,n]})$ is precisely the set of special columns of $P_{\SW}$ and $P_{\NE}$. Without loss of generality, suppose that column $n$ is special, or equivalently that $1\in I$. This means that 
\begin{equation*}
    \langle A_{[1,n-1]}\rangle \subset \langle h_{p+1}, \ldots, h_n \rangle.
\end{equation*} 
Then, S must include $n$ in order for the column span $\langle A_{S \cap [n]}\rangle$ to possibly contain $h_{p}$. On the other hand, $\barn \notin S$, because $S$ is admissible. 

Next, consider the columns $n-1,\overline{n-1}$, one of which must be special. The column span $\langle A_{S \cap [n]}\rangle$ contains the $n$-th column vector of $A$, which is a linear combination
\begin{equation*}
    \gamma_ph_p+\gamma_{p+1}h_{p+1},
\end{equation*}
where \emph{both} coefficients $\gamma_p,\gamma_{p+1}$ are non-zero, by Proposition \ref{prop:solve_for_*_saturated}. If $n-1$ is special and $n-1\notin S$, then the remaining vectors spanning $\langle A_{S \cap [n]}\rangle$ lie in 
\begin{equation*}
    \langle h_{p+2},\ldots,h_n\rangle,
\end{equation*}
so $h_{p}$ cannot lie in the span $\langle A_{S \cap [n]} \rangle$. Alternatively, if column $\overline{n-1}$ is special, then arguing as before shows that we need $\overline{n-1} \in S$ in order for $h_p\in \langle A_{S \cap \overline{[n]}} \rangle$. Iterating this argument gives the claim in general.

Finally, note that $S$ can contain at most one of $1,\bar1$. Without loss of generality, suppose that $1\notin S$. Then, a basis of $\langle A_{S \cap [n]} \rangle$ is given by the vectors in the special columns among $S\cap [n]$, each of which is a linear combination with non-zero coefficients of the vectors $h_{p'},h_{p'+1}$, where $p'=p,p+1,\ldots,n-1$. This column span cannot contain $h_p$, so we have reached a contradiction.
\end{proof}

\begin{corollary}\label{cor:split_g}
Adopting the setup of Proposition \ref{prop:x_in_P(I)}, we have
\begin{equation*}
    g_{\Lambda_I}(S) = g_{\Lambda_I}(S \cap [n])+ g_{\Lambda_I}(S \cap \overline{[n]}).
\end{equation*}
\end{corollary}

In particular, because 
\begin{equation*}
x(S) = x(S \cap [n]) + x(S \cap \overline{[n]}),
\end{equation*}
we have reduced Proposition \ref{prop:x_in_P(I)} to the case where $S\subset[n]$ or $S\subset\overline{[n]}$. 

\begin{proposition}\label{prop:S_ineq_positive}
    Proposition \ref{prop:x_in_P(I)} holds when $S\subset[n]$.
\end{proposition}

\begin{proof}
    Let $\rho_1<\cdots<\rho_r$ be the indices such that
\begin{equation*}
    n +1 -\lambda_{\rho_1},n +1 -\lambda_{\rho_2}, \ldots,n +1 -\lambda_{\rho_r} \in \SC(x)
\end{equation*}
are the special columns of $M_{I,I^c}$ in $[n]$ \emph{not} contained in $S$. It will be convenient also to set
\begin{equation*}
    \rho_0=0,\lambda_0=n+1\text{ and }\rho_{r+1}=s+1,\lambda_{s+1}=0.
\end{equation*}

Then, for every integer $i \in \{0, \ldots, r\}$, define the set
\begin{equation*}
    \widetilde{S}_i =[(n+1-\lambda_{\rho_{i}})+1, (n+1-\lambda_{\rho_{i+1}})-1].
\end{equation*}
By definition, the $\widetilde{S}_i$ are disjoint. We also have 
\begin{equation*}
    \bigcup_{i \in \{0, \ldots, r\}} \widetilde{S}_i = [n] \setminus \{n +1 -\lambda_{\rho_1}, n +1 -\lambda_{\rho_2}, \ldots, n +1 -\lambda_{\rho_r}\}.
\end{equation*}
Note that the set $\widetilde{S}_i$ can be empty for any $i \in \{1, \ldots, r\}$. On the other hand, $1$ is always an element of $\widetilde{S}_0$. Taking $A$ to be as in the proof of Lemma \ref{lem:split_rank} above, we observe that the subsets $\widetilde{S}_i$ satisfy the property that 
\begin{equation*}
    \langle A_{\widetilde{S}_i}\rangle \cap \langle A_{\widetilde{S}_j}\rangle = 0
\end{equation*}
for any $i \neq j$. This implies that 
\begin{equation*}
    \rk_{\Lambda_I}\left(\bigcup_{i \in \{0, \ldots, r\}} \widetilde{S}_i\right) = \rk_{\Lambda_I}(\widetilde{S}_0) + \ldots + \rk_{\Lambda_I}(\widetilde{S}_r).
\end{equation*}
Now, we can write
\begin{equation*}
    S = S_0 \cup S_1 \cup \ldots \cup S_r,
\end{equation*}
where $S_i=S\cap \widetilde{S}_i$, so that
\begin{equation*}
    \rk_{\Lambda_I}(S) = \rk_{\Lambda_I}(S_0) + \ldots + \rk_{\Lambda_I}(S_r).
\end{equation*}
It now suffices to show that $x(S_i)\le g_{\Lambda_I}(S_i)= \rk_{\Lambda_I}(S_i)$ for each $i$. 

By construction, $S_i$ contains all of the $\rho_{i+1}-\rho_i-1$ special columns in the interval $\widetilde{S}_i=[(n+1-\lambda_{\rho_{i}})+1, (n+1-\lambda_{\rho_{i+1}})-1]$. Then, we have
\begin{equation*}
    g_{\Lambda_I}(S_i) = \rk_{\Lambda_I} (S_i) = 
    \begin{cases}
        \rho_{i+1}-\rho_i-1 &\text{ if } S_i = \widetilde{S}_i\cap \SC(x),\\
        \rho_{i+1}-\rho_i&\text{ otherwise}
    \end{cases}
\end{equation*}
Indeed, by the non-vanishing of the $*$-entries in Proposition \ref{prop:solve_for_*_saturated}, the column vectors in the set of special columns $\SC(x)\cap  \widetilde{S}_i$ are linearly independent, and if $S_i$ contains at least one additional column vector in a non-special column, then $\langle A_{S_i}\rangle=\langle A_{\widetilde{S}_i}\rangle$. In the first case, because $x\in[0,1]^n$, we have
\begin{equation*}
    x(S_i)\le |S_i|= \rho_{i+1}-\rho_i-1=\rk_{\Lambda_I} (S_i),
\end{equation*}
so the needed inequality is immediate.

Suppose, from now on, that there is at least one element in $S_i$ that is not a special column. Then, we have $\rk_{\Lambda_I}(S_i) = \rk_{\Lambda_I}(\widetilde{S}_i) = \rho_{i+1}-\rho_i$ and $x(S_i) \le x(\widetilde{S}_i)$. Thus, it suffices to show that 
\begin{equation*}
    x(\widetilde{S}_i) \le \rho_{i+1}-\rho_i,
\end{equation*}
for $i=0,\ldots,r$. On the other hand, we have
\begin{align*}
    x(\widetilde{S}_i)=y_{n-\lambda_{\rho_{i+1}}}-y_{n+1-\lambda_{\rho_i}}\le \rho_{i+1}-\rho_i,
\end{align*}
by Definition \ref{def:I_associated_to_x}. This completes the proof.
\end{proof}

\begin{proposition}\label{prop:S_ineq_negative}
    Proposition \ref{prop:x_in_P(I)} holds when $S\subset\overline{[n]}$.
\end{proposition}

\begin{proof}
The argument is entirely dual to that of Proposition \ref{prop:S_ineq_positive}. Write 
\begin{equation*}
    \overline{n+1-\lambda'_{\rho_1}},\ldots,\overline{n+1-\lambda'_{\rho_r}},
\end{equation*}
where $\rho_1<\cdots<\rho_r$, for the special columns in $\overline{[n]}$ not contained in $S$, and write
\begin{equation*}
    S=S_0\cup S_1\cup\cdots\cup S_r,
\end{equation*}
where $S_i=\widetilde{S}_i\cap S=\overline{[(n+1-\lambda'_{\rho_{i}})+1,(n+1-\lambda'_{\rho_{i+1}})-1]}\cap S$.

The functions $\overline{x}(S):=x(S)+|S|$ and $\rk_{\Lambda_I}(S)=g_{\Lambda_I}(S)+|S|$ are both additive for the $S_i$, so it suffices to prove that
\begin{equation*}
\overline{x}(S)\le \rk_{\Lambda_I}(S)
\end{equation*}
for $S=S_i$. If $S_i$ only contains special columns, then
\begin{equation}\label{eq:x-bar-ineq}
\overline{x}(S_i)\le |S_i|=\rk_{\Lambda_I}(S_i),
\end{equation}
because $x_i\in[0,1]$, as needed. Otherwise, we have $\overline{x}(S_i)\le \overline{x}(\widetilde{S}_i)$ and $\rk_{\Lambda_I}(S_i)=\rk_{\Lambda_I}(\widetilde{S}_i)=\rho_{i+1}-\rho_i$, so it suffices to prove that
\begin{equation*}
\overline{x}(\widetilde{S}_i)\le \rho_{i+1}-\rho_i.
\end{equation*}
This, in turn, follows from Lemma \ref{lem:dual_ineqs}.
\end{proof}

Propositions \ref{prop:S_ineq_positive} and \ref{prop:S_ineq_positive} together imply Proposition \ref{prop:x_in_P(I)}, by Corollary \ref{cor:split_g}. Therefore, Proposition \ref{prop:polytope_union} is proven.

\subsection{Multiplicity 1}\label{sec:multiplicity_1}

\begin{proposition}\label{prop:multiplicity_1}
    Let $I\subset[n-1]$ be any subset. Let $\Xi$ (resp. $\Xi_I$) be the lattice generated by differences of vertices of $\cP(\Lambda)=[0,1]^n$ (resp. $\cP(\Lambda_I)$), see \S\ref{sec:prelim_degen}. Then, we have $\Xi=\Xi_I$.
\end{proposition}

\begin{proof}
    We must show that, for any $I$, we have $f_1,\ldots,f_n\in \Xi_I$. For a maximal admissible set $S\in\Adm_n$, and index $j\in[n]$, let $S(j)\in\Adm_n$ denote the result of replacing $j\in S$ with $\barj$, or vice versa. It suffices to show that, for all $j\in[n]$, there exists $S\in\cD(\Lambda_I)$ such that $S(j)\in\cD(\Lambda_I)$. Indeed, if this is the case, then $\chi(S)-\chi(S(j))=\pm f_j\in \Xi_I$. 
    
    As in \S\ref{sec:rank_ineqs}, for $p=1,\ldots,n$, let $h_p$ denote the column vector whose $p$-th entry is equal to 1 and whose other entries are zero. Let $A=\ssp^{-1}(\Lambda_I)$ be the unique matrix in $\cM_{I,I^c}$ whose row span is $\Lambda_I$, and for any subset $S$ of the columns of $A$, let $\langle A_S\rangle$ denote their span. Let $a^p\in \bC^{*}$ denote the leftmost $*$-entry in the $p$-th row of $A$. Let $\underline{\SC}$ be the set of all special columns of $A$.
    
    First, consider $j=1$. Then, let $S=\{\bar1\}\cup \underline{\SC}$. We have
    \begin{align*}
        \langle A_S\rangle &=\langle h_1,a^1h_1+h_2,\ldots,a^{n-1}h_{n-1}+h_n\rangle=\langle h_1,\ldots,h_n\rangle,\\
        \langle A_{S(1)}\rangle &=\langle a^1h_1+h_2,\ldots,a^{n-1}h_{n-1}+h_n,a^nh_n\rangle=\langle h_1,\ldots,h_n\rangle.
    \end{align*}
    Thus, $S$ has the needed property.

    Now, suppose that $j>1$, and suppose that column $j$ is special. Write $a^ph_p+h_{p+1}$ for the column vector in column $j$ and $a^{p'}h_{p'}$ for the column vector in column $\barj$, where $p'\le p$. Then, let $S=\{1\}\cup\underline{\SC}$. We have
    \begin{align*}
        \langle A_S\rangle &=\langle a^1h_1+h_2,\ldots,a^{n-1}h_{n-1}+h_n,a^nh_n\rangle=\langle h_1,\ldots,h_n\rangle,\\
        \langle A_{S(j)}\rangle &=\langle a^{p'}h_{p'}, a^1h_1+h_2,\ldots,a^{p-1}h_{p-1}+h_p\rangle\\
        &\quad \oplus 
        \langle a^{p+1}h_{p+1}+h_{p+2},\ldots,a^{n-1}h_{n-1}+h_n,a^nh_n\rangle=\langle h_1,\ldots,h_n\rangle.
    \end{align*}
    Thus, $S$ again has the needed property. If instead column $\barj$ is special, then take instead $S=\{\bar1\}\cup \underline{\SC}$.
    \end{proof}

\begin{remark}\label{rem:not_always_reduced}
    In contrast to the case of ordinary Grassmannians \cite[Proposition 1.2.15]{kapranov}, it is not true for arbitrary points $\Lambda\in\OG(n,2n+1)$ that the sublattice $\Xi$ associated to $\cP(\Lambda)$ equals the full lattice $\bZ^n$. For example, if
    \begin{equation*}
        \Lambda=\ssp\left(
        \begin{bmatrix}
            1 & 0 & 0 & 0 & a & b & 0 \\
            0 & 1 & 0 & 0 & c & 0 & -b\\
            0 & 0 & 1 & 0 & 0 & -c & -a
        \end{bmatrix}
        \right)\in\OG(3,7),
    \end{equation*}
    for $a,b,c\in\bC^{*}$, then $\cP(\Lambda)$ is the convex hull of the vectors $f_1,f_2,f_3,f_1+f_2+f_3\in\bR^3$, and $\Xi=\bZ\langle f_1+f_2,f_1+f_3,f_2+f_3\rangle$ has index 2 in $\bZ^3$. This example also appears in \cite[Example 5.1]{els}, but note that our column indexing is different.
\end{remark}

\begin{proof}[Proof of Theorem \ref{thm:degen}]
We first show that, for any non-saturated allowed pair $(J,J')$, in the degeneration $\pi_{J,J'}$ constructed in \S\ref{sec:degen}, the components $Z_{\Lambda_\ell},Z_{\Lambda_r}$ are the only ones in the special fiber $\cZ_0$. If this were not the case, then for the nicest $\Lambda_J\in Z_{J,J'}$, the complement
\begin{equation*}
    \cP(\Lambda_J)\backslash(\cP(\Lambda_\ell)\cup\cP(\Lambda_r))
\end{equation*}
would be contained in the complement of $\bigcup_{I\subset[n-1]}\cP(\Lambda_{I})$ in $\cP(\Lambda)$. However, Proposition \ref{prop:polytope_union} shows that this is impossible.

It follows that the algebraic cycle
\begin{equation*}
    [Z_\Lambda]-\sum_{I\subset [n-1]} c_I[\Sigma_{I,I^c}].
\end{equation*}
is homologically trivial, where the multiplicities $c_I$ have been determined to equal 1 in Proposition \ref{prop:multiplicity_1}. Accordingly, the limit components of all intermediate steps of the degeneration must also appear with multiplicity 1, or else the cycle above would be non-zero and effective, hence homologically non-trivial, as $\OG(n,2n+1)$ is projective. This completes the proof of Theorem \ref{thm:degen}. Theorem \ref{thm:class_formula} also follows.
\end{proof}

\appendix

\section{Generalities on $\OG(n,2n+1)$}

\subsection{Geometry of isotropic subspaces}

We adopt the notation of \S\ref{sec:OG}. Fix throughout an isotropic subspace $\Lambda\in\OG(n,2n+1)$. 

\begin{lemma}\label{find_nice_vector}
Let $q\in[n-1]$ be an integer. Suppose that
\begin{equation*}
    \dim(\Lambda\cap H_{\{1,\ldots,q,q+1,\overline{q+1}\}})\ge n-q.
\end{equation*}
Then, we have 
\begin{equation*}
\dim(\Lambda\cap H_{\{0,1,\ldots,q,q+1,\overline{q+1},\ldots,n,\overline{n}\}})\ge 1.
    \end{equation*}
\end{lemma}

\begin{proof}
There is nothing to check if $q=n-1$, so assume that $q\le n-2$. Observe that
\begin{align*}
    \dim(\Lambda\cap H_{\{1,\ldots,q,q+1,\overline{q+1},q+2,q+3,\ldots,n\}})&\ge 1, \\
    \dim(\Lambda\cap H_{\{1,\ldots,q,q+1,\overline{q+1},\overline{q+2},q+3,\ldots,n\}})&\ge 1,
\end{align*}
because $H_{\{q+2,q+3,\ldots,n\}}$ and $H_{\{\overline{q+2},q+3,\ldots,n\}}$ have codimension $n-q-1$ in $\bC^{2n+1}$. 

We claim that
    \begin{equation*}
    \dim(\Lambda\cap H_{\{1,\ldots,q,q+1,\overline{q+1},q+2,\overline{q+2},q+3,\ldots,n\}})\ge 1.
    \end{equation*}
Indeed, choose non-zero vectors
\begin{align*}
        v&\in \Lambda\cap H_{\{1,\ldots,q,q+1,\overline{q+1},q+2,q+3,\ldots,n\}}, \\
        w&\in \Lambda\cap H_{\{1,\ldots,q,q+1,\overline{q+1},\overline{q+2},q+3,\ldots,n\}}.
\end{align*}
so that, in coordinates,
\begin{align*}
v&=a_0e_0+a_{\bar1}e_{\bar1}+\cdots+a_{\overline{q}}e_{\overline{q}}+a_{\overline{q+2}}e_{\overline{q+2}}+a_{\overline{q+3}}e_{\overline{q+3}}+\cdots+a_{\overline{n}}e_{\overline{n}},\\
w&=b_0e_0+b_{\bar1}e_{\bar1}+\cdots+b_{\overline{q}}e_{\overline{q}}+b_{q+2}e_{q+2}+b_{\overline{q+3}}e_{\overline{q+3}}+\cdots+b_{\overline{n}}e_{\overline{n}}.\\
\end{align*}
Because $\Lambda$ is isotropic, we have $v\cdot v=w\cdot w=0$, hence $a_0=b_0=0$. Moreover, we have $v\cdot w=0$, hence $a_{\overline{q+2}}b_{q+2}=0$. Thus, one of $v,w$ lies in $H_{\{1,\ldots,q,q+1,\overline{q+1},q+2,\overline{q+2},q+3,\ldots,n\}}$, as needed.

Note also by the same argument (swapping the roles of $q+3$ and $\overline{q+3}$) that we have \emph{both} the inequalities
    \begin{align*}
    \dim(\Lambda\cap H_{\{1,\ldots,q,q+1,\overline{q+1},q+2,\overline{q+2},q+3,\ldots,n\}})&\ge 1,\\
    \dim(\Lambda\cap H_{\{1,\ldots,q,q+1,\overline{q+1},q+2,\overline{q+2},\overline{q+3},\ldots,n\}})&\ge 1.
    \end{align*}
Now, we iterate the argument: we can deduce further that
    \begin{equation*}
    \dim(\Lambda\cap H_{\{1,\ldots,q,q+1,\overline{q+1},q+2,\overline{q+2},q+3,\overline{q+3},\ldots,n\}})\ge 1,
    \end{equation*}
and so forth, until arriving at
    \begin{equation*}
    \dim(\Lambda\cap H_{\{1,\ldots,q,q+1,\overline{q+1},\ldots,n,\overline{n}\}})\ge 1.
    \end{equation*}

Finally, observe that
\begin{equation*}
    \Lambda\cap H_{\{1,\ldots,q,q+1,\overline{q+1},\ldots,n,\overline{n}\}}=\Lambda\cap H_{\{0,1,\ldots,q,q+1,\overline{q+1},\ldots,n,\overline{n}\}},
\end{equation*}
because $\Lambda$ is isotropic.
\end{proof}

\begin{lemma}\label{rank+1}
If $\dim(\Lambda\cap H_{\{1,\ldots,q\}})=n-q$, then $\dim(\Lambda\cap H_{\{1,\ldots,q,q+1,\overline{q+1}\}})\le n-q-1$.
\end{lemma}

\begin{proof}
 Assume for sake of contradiction $\dim(\Lambda\cap H_{\{1,\ldots,q,q+1,\overline{q+1}\}})=n-q$. By Lemma \ref{find_nice_vector}, we may choose a non-zero vector $v\in \Lambda\cap H_{\{0,1,\ldots,q,q+1,\overline{q+1},\ldots,n,\overline{n}\}}$. Note on the one hand that 
 \begin{equation*}
     \Lambda\subset \Lambda^{\perp}\subset v^{\perp},
 \end{equation*}
 and on the other hand that 
 \begin{equation*}
      H_{\{1,\ldots,q\}} = H_{\{0,1,\ldots,q,q+1,\overline{q+1},\ldots,n,\overline{n}\}}^{\perp} \subset v^{\perp}.
 \end{equation*}
 The subspaces $\Lambda,H_{\{1,\ldots,q\}}\subset v^{\perp}$ have dimensions $n,2n+1-q$, respectively, and $v^{\perp}$ has dimension $2n$. Thus,
 \begin{equation*}
     \dim(\Lambda\cap H_{\{1,\ldots,q\}})\ge (n+(2n+1-q))-2n=n-q+1,
 \end{equation*}
 which is a contradiction.
\end{proof}

\begin{corollary}\label{cor:rank+j}
    Let $S\subset \wtn$ be an admissible subset. Let $J\subset[n]$ be such that $J\cup \overline {J}$ is disjoint from $S$. Write $\ell=|J|$. Then, for some admissible subset $J_+\subset J\cup \overline {J}$ of size $\ell$ (the largest possible), we have
    \begin{equation*}
        \rk_\Lambda(S\cup J_+) = \rk_\Lambda(S)+\ell
    \end{equation*}
\end{corollary}

If $A$ is any matrix whose rows form a basis of $\Lambda$, then recall from Definition \ref{def:rank} that $\rk_\Lambda(S)$ is the rank of the submatrix $A_S$ consisting of columns of $A$ indexed by $S$. Equivalently, we have $\rk_\Lambda(S)=n-\dim(\Lambda\cap H_S)$. On the level of matrices, Corollary \ref{cor:rank+j} says that, for each $j\in J$, one can make a choice of either $j$ or $\barj$ (to comprise $J_+$) so that the corresponding $\ell$ column vectors of $A$ are independent from each other and from the column span of $A_S$.

\begin{proof}
We may assume that $S=\{1,2,\ldots,q\}$ by permuting coordinates. We may replace $A_S$ with a submatrix consisting of independent columns, so reduce to the case where $\rk(A_S)=|S|=q$. We may similarly assume that $J=\{q+1,\ldots,q+\ell\}$ by permuting coordinates. Then, the claim follows from Lemma \ref{rank+1} by induction on $\ell$.
\end{proof}

\begin{corollary}\label{cor:rank_in_terms_of_M}
    For any $S\in \Adm_n$, we have
    \begin{equation*}
        \rk_\Lambda(S)=\max_{S'\in\cD(\Lambda)}|S'\cap S|.
    \end{equation*}
\end{corollary}

\begin{proof}
As above, we interpret $\rk_\Lambda(S)$ as the rank of the matrix $A_S$. For any $S'\in \cD(\Lambda)$, we have
\begin{equation*}
    \rk(A_S)\ge \rk(A_{S\cap S'})=|S'\cap S|,
\end{equation*}
so $\rk(A_S)\ge \max_{S'\in\cD(\Lambda)}|S'\cap S|$. 

Conversely, let $S_{-}\subset S$ be a maximal subset of linearly independent columns of $A_S$, so that $\rk(A_S)=\rk(A_{S_{-}})=|S_{-}|$. Replacing $S$ with $S_{-}$ and taking
\begin{equation*}
    J=[n]\backslash \left((S_{-}\cap [n]) \cup \overline{(S_{-}\cap \overline{[n]})}\right)
\end{equation*}
in Corollary \ref{cor:rank+j}, there exists $S'\in\cD(\Lambda)$ such that $S'\supset S_{-}$, namely, $S'=S_{-}\cup J_{+}$. In particular, we have
\begin{equation*}
    \rk(A_S)=|S_{-}|\le |S'\cap S|\le \max_{S'\in\cD(\Lambda)}|S'\cap S|,
\end{equation*}
as needed.
\end{proof}

In the next section, we will also need the following property of the function $g_\Lambda(S)$, also defined in Definition \ref{def:rank}.

\begin{corollary}\label{cor:bisubmodular}
    The function $g_\Lambda(S):=\rk_\Lambda(S)-|S\cap \overline{[n]}|$ is \emph{bisubmodular}. That is, for any two $S_1,S_2\in\Adm_n$, we have
    \begin{equation*}
        g_\Lambda(S_1)+g_\Lambda(S_2)\ge g_\Lambda(S_1\cap S_2)+g_\Lambda(S_1 \sqcup S_2),
    \end{equation*}
    where $S_1\sqcup S_2:=\{a\in S_1\cup S_2:\overline{a}\notin S_1\cup S_2\}$.
\end{corollary}

\begin{proof}
    We have:
    \begin{align*}
        |(S_1\cap S_2)\cap \overline{[n]}|+|(S_1\cup S_2)\cap \overline{[n]}|&=|S_1\cap\overline{[n]}|+|S_2\cap \overline{[n]}|,\\
        \rk_\Lambda(S_1)+\rk_\Lambda(S_2)&\ge \rk_\Lambda(S_1\cap S_2)+\rk_\Lambda(S_1 \cup S_2),\\
        \rk_\Lambda(S_1\cup S_2)-\rk_\Lambda(S_1\sqcup S_2)&\ge |\{j:j,\barj\in S_1\cup S_2\}|.
    \end{align*}
    The first two lines are elementary, and the third follows from Corollary \ref{cor:rank+j}. The claim follows from adding the three lines.
\end{proof}

\subsection{Matroid polytopes}\label{sec:polytopes}

We adopt the notation of \S\ref{sec:delta}. In this section, we prove:

\begin{theorem}\label{thm:polytopes_same}
Let $\Lambda\in\OG(n,2n+1)$ be any point, and let $\cX(\Lambda)$ be the convex hull in $\mathbb{R}^n$ of the vectors
\begin{equation*}
    \chi(S'):=\sum_{q\in [n]\cap S'}f_q\in\mathbb{R}^n.
\end{equation*}
raniging over all $S'\in\cD(\Lambda)$. Then, $\cX(\Lambda)=\cP(\Lambda)$, as where $\cP(\Lambda)$ is defined in Definition \ref{def:polytope}.
\end{theorem}

Theorem \ref{thm:polytopes_same} and the proof we provide are likely well-known, but we were not able to find a self-contained reference. One inclusion, Proposition \ref{prop:X_in_P}, is easy.

\begin{lemma}\label{lem:chi(S')(S)_formula}
    Let $S,S'\subset \wtn$ be admissible subsets, and assume furthermore that $S'\in\mAdm_n$ is maximal. Write $x=\chi(S')\in\bR^n$. Then, we have
    \begin{equation*}
        x(S)=|S\cap S'|-|S\cap \overline{[n]}|,
    \end{equation*}
    where $x(S)$ is defined in Definition \ref{def:x(S)}.
\end{lemma}

\begin{proof}
 We have 
    \begin{align*}
        x(S) &= |S' \cap [n] \cap S | - | S' \cap[n] \cap \overline{S}|\\
        &=(|S' \cap [n] \cap S |+|S\cap \overline{[n]}|)-(| S' \cap[n] \cap \overline{S}|+|S\cap \overline{[n]}|)\\
        &=(|S' \cap [n] \cap S |+|S|-|S\cap [n]|)-(| S' \cap[n] \cap \overline{S}|+|S\cap \overline{[n]}|)\\
        &=(| S | - | \overline{S'}\cap[n]\cap S |)-(| S' \cap[n] \cap \overline{S}|+|S\cap \overline{[n]}|)\\
        &=| S | - (| \overline{S'}\cap[n]\cap S |+| \overline{S'} \cap\overline{[n]} \cap S|)-|S\cap \overline{[n]}|\\
        &=(| S | - |S\cap \overline{S'}|)-|S\cap \overline{[n]}|\\
        &=|S\cap S'|-|S\cap \overline{[n]}|.
    \end{align*}
    The first line follows from unwinding the definition of $x(S)=(\chi(S'))(S)$. The remaining calculation follows from elementary set-theoretic considerations, using that $S\in\Adm_n$ and $S'\in\mAdm_n$.
\end{proof}

\begin{proposition}\label{prop:X_in_P}
We have $\cX(\Lambda)\subset\cP(\Lambda)$.
\end{proposition}

\begin{proof}
    Because both $\cP(\Lambda)$ is convex, it suffices to prove that, for all $S'\in \cD(\Lambda)$, we have $\chi(S')\in \cP(\Lambda)$. Write $x=\chi(S')$ as in Lemma \ref{lem:chi(S')(S)_formula}, and let $S\subset\wtn$ be any admissible subset. By Lemma \ref{lem:chi(S')(S)_formula}, we have
    \begin{align*}
        x(S) &= |S\cap S'|-|S\cap \overline{[n]}|\\
        &\le\rk_{\Lambda}(S)-|S\cap \overline{[n]}|\\
        &=g_\Lambda(S).
    \end{align*}
The inequality $|S\cap S'|\le \rk_\Lambda(S)$ follows from the fact that $S'\in\cD(\Lambda)$. Indeed, if $A$ is a matrix whose rows form a basis for $\Lambda$, then the square submatrix $A_{S'}$ comprised of columns indexed by $S'$ has full rank, so the submatrix $A_{S\cap S'}$ of $A_S$ is comprised of $|S'\cap S|$ independent columns. As $x(S)\le g_\Lambda(S)$ for all admissible $S$, the point $x=\chi(S')$ satisfies all of the defining inequalities of $\cP(\Lambda)$. This proves the Proposition.
\end{proof}

To establish the opposite inclusion, we employ the following standard strategy (see e.g. \cite[\S 4.1.2-4.1.3]{lrs}): let $x\in\cP(\Lambda)$ be a vertex, so that a non-empty subset of the inequalities $x(S)\le g_\Lambda(S)$ are in fact equalities. By finding many admissible $S$ for which $x(S)= g_\Lambda(S)$ (Lemma \ref{lem:chain}), we will show that $x\in\{0,1\}^n$, and deduce that $x=\chi(S')$ for some $S'\in\cD(\Lambda)$.

\begin{lemma}\label{lem:F_closed}
    Let $x \in \cP(\Lambda)$ be any point. Then, the collection of subsets
    \begin{align*}
        \mathcal{F}_x = \{S \in\Adm_n: x(S) = g_\Lambda(S)\}.
    \end{align*}
    is closed under the operations $\cap,\sqcup$. That is, we have
    \begin{align*}
     S_1,S_2 \in \mathcal{F}_x \Rightarrow S_1 \cap S_2 \in \mathcal{F}_x \text{ and } S_1 \sqcup S_2 \in \mathcal{F}_x.
    \end{align*}
\end{lemma}
\begin{proof}
Immediate from bi-submodularity of the function $g_\Lambda$, Corollary \ref{cor:bisubmodular}, and the fact that
\begin{equation*}
    x(S_1)+x(S_2)=x(S_1\cap S_2)+x(S_1\sqcup S_2).
\end{equation*}
\end{proof}

\begin{lemma}\label{lem:chain}
    Let $x \in \cP(\Lambda)$ be a vertex, and define $\cF_x$ as in Lemma \ref{lem:F_closed}. Then, there exists a chain of admissible subsets
    \begin{equation*}
     S_1 \subset S_2\cdots\subset S_n\subset \wtn,
    \end{equation*}
    with $|S_j|=j$ for each $j$, and $S_j\in\cF_x$.
\end{lemma}

\begin{proof}
Consider any maximal chain
\begin{equation*}
     \cC=(S_{j_1} \subset S_{j_2}\cdots\subset S_{j_m}\subset \wtn)
 \end{equation*}
of non-empty admissible subsets lying in $\cF_x$, by which we mean that no further admissible subset $S\in\cF_x$ can be inserted into $\cC$. We choose the indices so that $|S_{j_k}|=j_k$.

Write $j=j_1$. We claim that $j=1$; suppose instead that $j>1$. Without loss of generality, we assume for convenience that $S_{j}=[j]$. We claim that any $S\in \cF_x$ must have
\begin{equation*}
    S\cap ([j]\cup\overline{[j]})\in\{\emptyset,[j],\overline{[j]}\}.
\end{equation*}
Indeed, by Lemma \ref{lem:F_closed}, we have $S\cap [j]=S\cap S_j\in \cF_x$. Unless $S\cap [j]=\emptyset$ or $S\cap [j]=[j]$, the proper subset $S\cap S_j\subset S_j$ may be appended to $\cC$, a contradiction. If $S\cap [j]=[j]$, then because $S$ is admissible, in fact $S\cap ([j]\cup\overline{[j]})=[j]$. If instead $S\cap [j]=\emptyset$, then $S\cap ([j]\cup\overline{[j]})=S\cap\overline{[j]}$. In this case, we have
\begin{equation*}
    (S\sqcup S_j)\cap S_j=[j]-(\overline{S\cap\overline{[j]}})\in \cF_x.
\end{equation*}
We again obtain a contradiction unless $S\cap\overline{[j]}\in \{\emptyset,\overline{[j]}\}$, as needed.

Fix any vector $v=v_1f_1+\cdots+v_jf_j\subset\langle f_1,\ldots,f_j\rangle \subset \bR^n$ for which $v_1+\cdots+v_j=0$. We claim that there exists an $\varepsilon>0$ such that $x+\varepsilon v\in \cP(\Lambda)$. Indeed, taking $\varepsilon$ sufficiently small preserves the (finitely many) \emph{strict} inequalities $(x+\varepsilon v)(S)<g_\Lambda(S)$ for any admissible $S\notin\cF_x$. Furthermore, all \emph{equalities} $(x+\varepsilon v)(S)=g_\Lambda(S)$ for $S\in\cF_x$ remain true for any $\varepsilon$, because $S\cap ([j]\cup\overline{[j]})\in\{\emptyset,[j],\overline{[j]}\}$ and $v_1+\cdots+v_j=0$ together imply that $v(S)=0$.

When $j>1$, the existence of such an $\varepsilon$ for any $v$ contradicts the fact that $x\in\cP(\Lambda)$ is a vertex. Indeed, we deduce that $x$ may be perturbed non-trivially inside $\cP(\Lambda)$ along any tangent direction in the linear span $\langle f_1-f_2,\ldots,f_1-f_j\rangle$, which cannot hold for a vertex of $\cP(\Lambda)$. Therefore, we have $j=j_1=1$.

Similar arguments show that $j_2-j_1=1$, that $j_3-j_2=1$, and so on, until $j_m=n$ (so $m=n$). Therefore, the maximal chain $\cC$ must be of length $n$, and the proof is complete.
\end{proof}

\begin{proof}[Proof of Theorem \ref{thm:polytopes_same}]
    By Proposition \ref{prop:X_in_P}, it suffices to show that $\cP(\Lambda)\subset\cX(\Lambda)$. It suffices in turn to show that any vertex $x \in \cP(\Lambda)$ lies in $\cX(\Lambda)$. Fix such a vertex $x$.

    Choose a chain of subsets $S_j$ as in Lemma \ref{lem:chain}. In particular, we have $x(S_j)=g_\Lambda(S_j)$ for $j=1,\ldots,n$. By considering the equations $x(S_j)=g_\Lambda(S_j)$ in increasing order of $j$, we see that there is a \emph{unique} point $x\in\bR^n$ with $x(S_j)=g_\Lambda(S_j)$ for all $j$, and that $x\in\{0,1\}^n$. Therefore, the given vertex $x \in \cP(\Lambda)$ must equal $\chi(S)$ for some maximal admissible set $S$. 

    In particular, we have $\chi(S)(S)\le g_{\Lambda}(S)$. By Lemma \ref{lem:chi(S')(S)_formula}, this implies that $n=|S|\le \rk_\Lambda(S)$. Thus, $S\in\cD(\Lambda)$, so $x=\chi(S)\in \cX(\Lambda)$.
\end{proof}

\bibliographystyle{alpha} 
\bibliography{OG_v3.bib}

\end{document}